\documentclass{eajam}
\renewcommand\thisnumber{x}
\renewcommand\thisyear {200x}
\renewcommand\thismonth{xxx}
\renewcommand\thisvolume{xx}

\usepackage{charter}
\usepackage[charter]{mathdesign}

\newcommand{\R}{\mathbb{R}}
\newcommand{\C}{\mathbb{C}}

\newcommand{\vx}{\mathbf{x}}
\newcommand{\rmd}{\mathrm{d}}

\newcommand{\Phit}[1]{{\Phi_{#1}^\tau}}
\newcommand{\St}[1]{{\mathcal{S}_{#1}^\tau}}
\newcommand{\psin}[1]{\psi^{#1}}
\newcommand{\psihn}[1]{\psi^{\langle #1 \rangle}}
\newcommand{\figheight}{0.25}

\makeatletter
\renewenvironment{proof}[1][\relax]{\par
	\normalfont \topsep6\p@\@plus6\p@\relax
	\trivlist
	\item[\hskip\labelsep\itshape
	\ifx#1\relax \proofname\else\proofname{} of #1\fi\@addpunct{.}]\ignorespaces
}{%
	\popQED\endtrivlist\@endpefalse
}
\makeatother

\begin{document}

	\markboth{W.~Bao, B.~Lin, Y.~Ma and C. Wang}{An extended Fourier pseudospectral method for GPE}
	\title{An extended Fourier pseudospectral method for the Gross-Pitaevskii equation with low regularity potential}
	
	
	\author[W. Bao, B.~Lin, Y.~Ma and C.~Wang]{Weizhu Bao\affil{1}, Bo Lin\affil{1}, Ying Ma\affil{2} and Chushan Wang\affil{1}\comma\corrauth}
	\address{\affilnum{1}\ Department of Mathematics, National University of Singapore, Singapore 119076\\
		\affilnum{2}\ Department of Mathematics, Faculty of Science, Beijing University of Technology, Beijing 100124, China}
	%
	%
	\emails{{\tt matbaowz@nus.edu.sg} (W. Bao), {\tt linbo@u.nus.edu} (B. Lin), {\tt maying@bjut.edu.cn} (Y. Ma), {\tt e0546091@u.nus.edu} (C. Wang)}
	%
	\begin{abstract}
		We propose and analyze an extended Fourier pseudospectral (eFP) method for the spatial discretization of the Gross-Pitaevskii equation (GPE) with low regularity potential by treating the potential in an extended window 
		for its discrete Fourier transform. The proposed eFP method maintains optimal convergence rates with respect to the regularity of the exact solution even if the potential is of low regularity and enjoys similar computational cost as the standard Fourier pseudospectral method, and thus it is both efficient and accurate. Furthermore, similar to the Fourier spectral/pseudospectral methods, the eFP method can be easily coupled with different popular temporal integrators including finite difference methods, time-splitting methods and exponential-type integrators. Numerical results are presented to validate our optimal error estimates and to demonstrate that they are sharp as well as to show its efficiency in practical computations. 
	\end{abstract}
	
	
	
	\maketitle
	
	
	\section{Introduction}\label{sec1}
	The Gross-Pitaevskii equation (GPE), as a particular case of the 
	nonlinear Schr\"odinger equation (NLSE) with cubic nonlinearity, 
	is derived from the mean-field approximation of many-body problems in quantum physics and chemistry, which is widely adopted in modeling and simulation of Bose-Einstein condensation (BEC) \cite{review_2013,ESY,NLS}. In this paper, we consider the following time-dependent GPE \cite{review_2013,ESY,NLS}
	\begin{equation}\label{NLSE}
		\left\{
		\begin{aligned}
			&i \partial_t \psi(\vx, t) = -\Delta \psi(\vx, t) + V(\vx) \psi(\vx, t) +  \beta |\psi(\vx, t)|^2 \psi(\vx, t), \quad  \vx \in \Omega, \  t>0, \\
			&\psi(\vx, 0) = \psi_0(\vx), \quad \vx \in \overline{\Omega},
		\end{aligned}
		\right.
	\end{equation}
	where $t$ is time, $\vx\in \R^d$ ($d=1, 2, 3$) is the spatial coordinate, $ \psi:=\psi(\vx, t) $ is a complex-valued wave function, and $ \Omega = \Pi_{i=1}^d (a_i, b_i) \subset \R^d $ is a bounded domain equipped with periodic boundary condition. Here, $ V:=V(\vx)$ is a time-independent real-valued potential and $\beta \in \R$ is a given parameter that characterizes the nonlinear interaction strength.  
	
	Usually, the potential $V$ is a smooth function which is chosen as either the harmonic potential, e.g. $V(\vx) = |\vx|^2/2$, or an optical lattice potential, e.g. $V(\vx) = \sum_{j=1}^d A_j \cos(L_jx_j) $ in $d$-dimensions 
	with $A_j, L_j \ (j=1,\ldots, d)$ being some given real-valued constants.  
	For the GPE \eqref{NLSE} with sufficiently smooth potential, many accurate and efficient temporal discretizations have been proposed and analyzed in last two decades, including the finite difference time domain (FDTD) method \cite{FD,bao2013,review_2013,Ant}, the exponential wave integrator (EWI) \cite{bao2014,ExpInt,SymEWI}, the time-splitting method \cite{bao2003JCP,BBD,lubich2008,review_2013,schratz2016,Ant,splitting_low_reg,RCO_SE,bao2023_semi_smooth}, and the low regularity integrator (LRI) \cite{LRI,LRI_sinum,LRI_error,LRI_JEMS,LRI_general,tree1,tree2,bronsard2023}. Generally, these temporal discretizations are followed by a spatial discretization, such as finite difference methods, finite element methods or Fourier spectral/pseudospectral methods, to obtain a full discretization for the GPE 
	\eqref{NLSE}. Among them, the time-splitting Fourier pseudospectral (TSFP) method is the most popular one due to its efficient implementation and spectral accuracy in space as well as the preservation of many dynamical properties of the GPE \eqref{NLSE} in the fully discretized level \cite{Ant,review_2013}. 
	
	On the contrary, in many physics applications, low regularity potential is also widely 
	incorporated into the GPE. Typical examples include the square-well potential or step potential \cite{poten_step_1,poten_step_2,poten_box}, which are discontinuous;  narrow potential barriers \cite{poten_narrow,poten_narrow2} and power-law potential 
	\cite{poten_power_1,poten_power_2}, which may have large or even unbounded derivatives; and the random potential or disorder potential in the study of Anderson localization \cite{poten_Josephson,poten_anderson}, which could be very rough. Some of them in one dimension (1D) and two dimensions (2D) are plotted in Figure \ref{fig:poten}. 
	
	\begin{figure}[htbp]
		\centering
		{\includegraphics[height=\figheight\textheight]{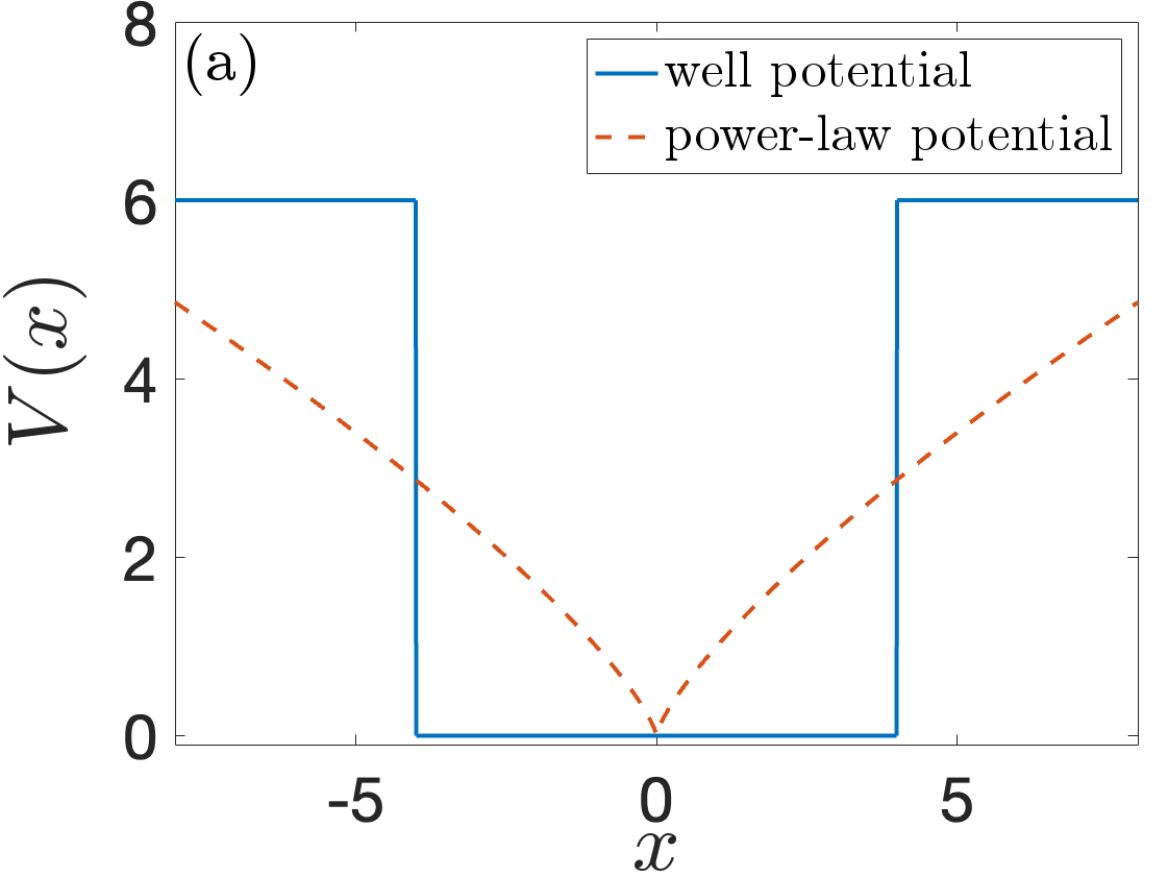}}\hspace{1em}
		{\includegraphics[height=\figheight\textheight]{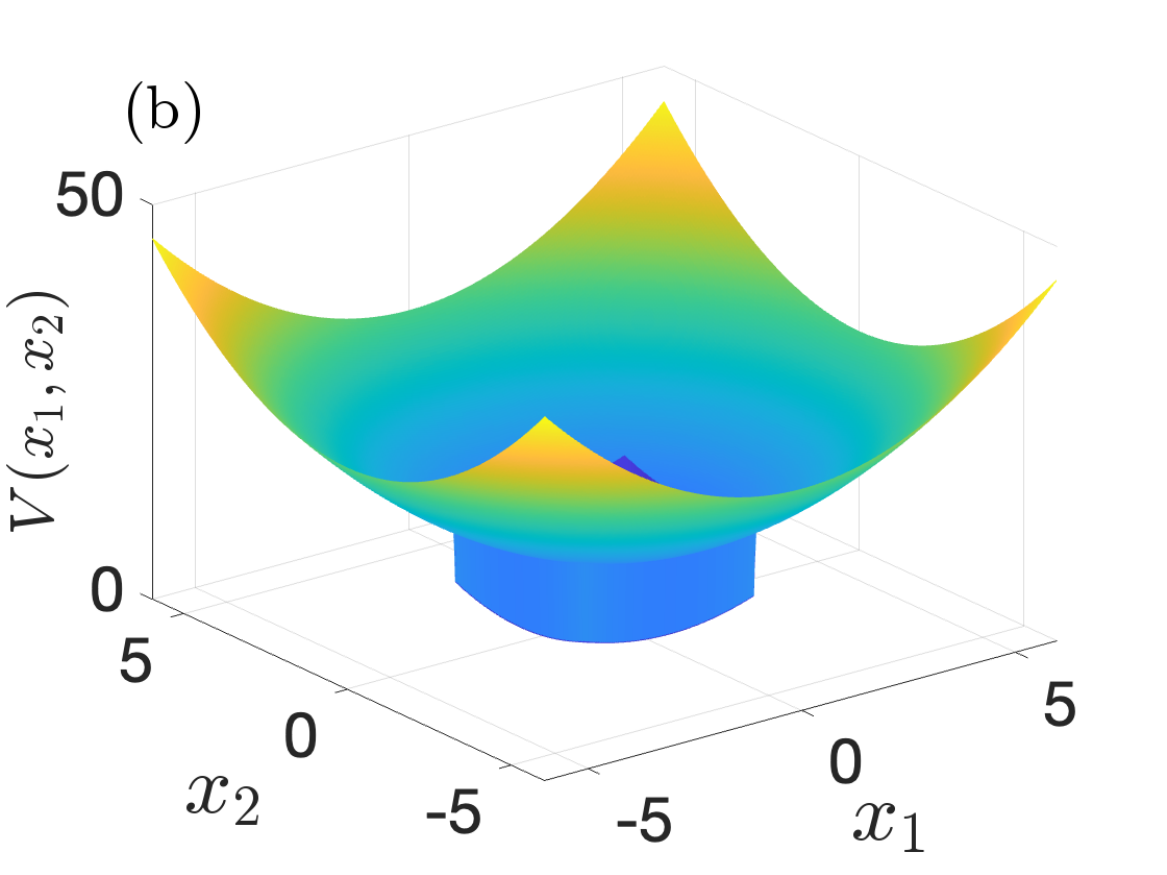}}
		\caption{Examples of low regularity potential: (a) square-well potential and power-law potential with order $0.75$ in 1D, and (b) square-well potential combined with a harmonic potential in 2D.} 
		\label{fig:poten}
	\end{figure}
	
	When considering the GPE with low regularity potential, most of the aforementioned full-discretization 
	methods are still applicable. However, their performance may significantly deviate from the smooth cases, leading to possible order reduction in both time and space. Recently, much attention has been paid to the error analysis of those methods for the GPE with low regularity potential and/or the NLSE with low regularity nonlinearity. For details, we refer to \cite{henning2017} for the FDTD method, \cite{bao2023_EWI} for the EWI, \cite{bao2023_improved,bao2023_semi_smooth,zhao2021} for the time-splitting method, and \cite{bronsard2022,bronsard2023,tree2} for the LRI. Among these methods, under the assumption of $L^\infty$-potential, for a general $H^2$-solution (which is the best situation that one can expect on the regularity of the exact solution due to the low regularity of potential), the optimal $L^2$-norm error bound -- first order in time and second order in space -- can only be proved for the EWI and the time-splitting method with the Fourier spectral (FS) method for spatial discretization \cite{bao2023_EWI,bao2023_improved}.  According to both theoretical and numerical results in \cite{bao2023_EWI,bao2023_improved}, it is essential to use the FS method rather than the Fourier pseudospectral (FP) method to discretize the EWI and time-splitting methods in space in order to get
	optimal spatial convergence. However, the FS method cannot be efficiently implemented in practice due to the difficulty in the exact evaluation of the Fourier integrals, and the usual strategy of using quadrature rules is extremely time-consuming to obtain the required accuracy due to the low regularity of potential. In fact, if the integrals are not approximated with sufficient accuracy, severe order reduction is observed in our numerical experiments. 
	More importantly, when applying time-splitting methods with low regularity potential, sub-optimal spatial convergence will even destroy the optimal convergence order in time since a CFL-type condition must be satisfied (see Section \ref{sec:numerical experiments} and \cite{bao2023_improved} for more details). In other words, the time-splitting Fourier spectral (TSFS) method can achieve optimal convergent rates -- an advantage in accuracy, but it is extremely expensive in practical computation since one needs to evaluate numerically the Fourier coefficients very accurately -- a disadvantage in efficiency. On the contrary, 
	the TSFP method is very efficient due to FFT -- an advantage in efficiency, 
	but it suffers from convergent rate reduction -- a disadvantage in accuracy.  
	
	The main aim of this paper is to propose an extended Fourier pseudospectral (eFP) method for the spatial discretization 
	of the GPE with low regularity potential. The proposed eFP method enjoys the advantages of the optimal 
	convergent rates of the FS method -- accuracy -- and the efficient implementation of the FP method -- efficiency, 
	and at the same time, it avoids the disadvantages of the expensive computational cost of the FS method and 
	the convergent rates reduction of the FP method. In summary, the proposed eFP method has the same accuracy as the FS method and similar computational cost as the FP method for the GPE with low regularity potential, and thus
	it has advantages in both accuracy and efficiency compared to those existing FS and FP methods! 
	In addition, it can be flexibly coupled with different temporal integrators such as FDTD methods, time-splitting methods, and exponential-type integrators. For simplicity, we take the time-splitting methods as a temporal integrator to present the eFP method for the GPE. We rigorously prove the optimal error bounds of the proposed time-splitting extended Fourier pseudospectral (TSeFP) methods for the GPE with low regularity potential. Extensions of the proposed TSeFP method by replacing the time-splitting methods with other temporal integrators are straightforward and their error estimates can be established similarly.   
	
	The key ingredient of the eFP method is to approximate the nonlinearity by Fourier interpolation and to approximate the potential by Fourier projection, both of which maintain the optimal approximation error. 
	Moreover, the Fourier interpolation and projection also filter out high frequencies, which makes it possible to fast compute exact Fourier coefficients of the product terms involving low regularity potential and functions of fixed finite frequency. We remark here that the eFP method can be viewed as an efficient implementation of the FS method in the presence of low regularity potential without any loss of accuracy. Finally, we would also like to mention some related ideas in \cite{bao2023_EWI,zhao_wang}. 
	
	The rest of the paper is organized as follows. In Sections 2 and 3, we present the time-splitting extended Fourier pseudospectral method and prove the optimal error bounds. Numerical results are reported in Section 4 to confirm our error estimates. Finally, some conclusions are drawn in Section 5. Throughout the paper, we adopt standard Sobolev spaces as well as their corresponding norms and denote by $ C $ a generic positive constant independent of the mesh size $ h $ and time step size $ \tau $, and by $ C(\alpha) $ a generic positive constant depending only on the parameter $ \alpha $. The notation $ A \lesssim B $ is used to represent that there exists a generic constant $ C>0 $, such that $ |A| \leq CB $. 
	
	\section{A time-splitting extended Fourier pseudospectral method}
	In this section, we begin with the FS method and then present the extended Fourier pseudospectral (eFP) method to discretize the first-order Lie-Trotter splitting for the GPE with low regularity potential. As mentioned before, one can also combine the eFP method with high-order splitting schemes or other temporal discretizations such as FDTD methods and EWIs. 
	
	
	For simplicity of the presentation and to avoid heavy notations, we only present the numerical schemes in one dimension (1D) and take $ \Omega = (a, b) $. Generalizations to two dimensions (2D) and three dimensions (3D) are straightforward. We define periodic Sobolev spaces as (see, e.g. \cite{bronsard2022}, for the definition in phase space)
	\begin{equation*}
		H_\text{per}^m(\Omega) := \{\phi \in H^m(\Omega) : \phi^{(k)}(a) = \phi^{(k)}(b), \ k=0, \cdots, m-1\}, \quad m \geq 1. 
	\end{equation*}
	
	The operator splitting techniques are based on a decomposition of the flow of \eqref{NLSE}: 
	\begin{equation}
		\partial_t \psi = A(\psi) + B(\psi),
	\end{equation}
	where 
	\begin{equation}
		A (\psi) = i \Delta \psi, \qquad B(\psi) = -i (V + f(|\psi|^2)) \psi,
	\end{equation}
	with 
	\begin{equation}
		f(\rho) := \beta \rho, \qquad \rho \geq 0. 
	\end{equation}
	Then the NLSE \eqref{NLSE} can be decomposed into two sub-problems. The first one is
	\begin{equation}
		\left\{
		\begin{aligned}
			&\partial_t \psi(x, t) = A(\psi) = i \Delta \psi(x, t), \quad x \in \Omega, \quad t>0, \\
			&\psi(x, 0) = \psi_0(x), \quad x \in \overline{\Omega},
		\end{aligned}
		\right.
	\end{equation}
	which can be formally integrated exactly in time as
	\begin{equation}\label{eq:linear_step}
		\psi(\cdot, t) = e^{i t \Delta} \psi_0(\cdot), \qquad t \geq 0.
	\end{equation}
	The second one is
	\begin{equation}
		\left\{
		\begin{aligned}
			&\partial_t \psi(x, t) = B(\psi) = -iV(\vx)\psi(x, t) -i f(|\psi(x, t)|^2)\psi(x, t), \quad  t>0, \\
			&\psi(x, 0) = \psi_0(x), \quad x \in \overline{\Omega},
		\end{aligned}
		\right.
	\end{equation}
	which, by noting
	$|\psi(x, t)|=|\psi_0(x)|$ for $t \geq 0$, can be integrated exactly in time as
	\begin{equation}\label{eq:nonl_step}
		\psi(x, t) = \Phi_B^t (\psi_0)(x) := \psi_0(x)e^{- i t (V(x) + f(|\psi_0(x)|^2))}, \quad x \in \overline{\Omega}, \quad t \geq 0. 
	\end{equation}
	
	Choose a time step size  $ \tau > 0 $, denote time steps as $ t_n = n \tau $ for $ n = 0, 1, ... $, and let $ \psi^{[n]}(\cdot) $ be the approximation of $ \psi(\cdot, t_n) $ for $ n \geq 0 $. Then a first-order semi-discretization of the NLSE \eqref{NLSE} via the Lie-Trotter splitting is given as
	\begin{equation}\label{eq:LT}
		\psi^{[n+1]} = e^{i \tau \Delta} \Phi_B^\tau\left (\psi^{[n]} \right ), \quad n \geq 0, 
	\end{equation}
	with $ \psi^{[0]}(x) = \psi_0(x) $ for $x\in\overline{\Omega}$. 
	
	\subsection{The FS method and the FS method with quadrature}
	
	Then we further discretize the semi-discretization \eqref{eq:LT} in space by the FS method to obtain a fully discrete scheme. 
	Choose a mesh size $h = (b-a)/N$ with $N$ being a positive even integer, and denote the grid points as \vspace{-0.5em}
	\begin{equation}
		x_j = a + jh, \qquad 0 \leq j \leq N. \vspace{-1em}
	\end{equation}
	Define the index sets
	\begin{equation}\vspace{-1em}
		\mathcal{T}_N = \left\{-\frac{N}{2}, \cdots, \frac{N}{2}-1 \right\}, \qquad \mathcal{T}_N^0 = \{0, 1, \cdots, N\}. 
	\end{equation}
	Denote
	\begin{align}
		&X_N = \text{span}\left\{e^{i \mu_l(x - a)}: l \in \mathcal{T}_N\right\}, \quad \mu_l = \frac{2 \pi l}{b-a}, \\
		&Y_N = \{v=(v_0, v_1, \cdots, v_N)^T \in \C^{N+1}: v_0 = v_N\}. 
	\end{align}
	Let $ P_N:L^2(\Omega) \rightarrow X_N $ be the standard $ L^2 $-projection onto $ X_N $ and $I_N: Y_N \rightarrow X_N$ be the standard Fourier interpolation operator as 
	\begin{equation}
		\vspace{-0.25em}
		(P_N u)(x) = \sum_{l \in \mathcal{T}_N} \widehat u_l e^{i \mu_l(x - a)}, \quad (I_N v)(x) = \sum_{l \in \mathcal{T}_N} \widetilde{v}_l e^{i \mu_l(x - a)},  \qquad x \in \overline{\Omega} = [a, b], 
	\end{equation}
	where $ u \in L^2(\Omega) $, $v \in Y_N$ and
	\begin{equation}\label{eq:hat}
		\vspace{-0.25em}
		\widehat{u}_l = \frac{1}{b-a} \int_a^b u(x) e^{-i \mu_l (x-a)} \rmd x, \quad \widetilde{v}_l = \frac{1}{N} \sum_{j=0}^{N-1} v_j e^{- i \mu_l (x_j - a)}, \qquad l \in \mathcal{T}_N. 
	\end{equation}
	Let $ \psi^n(\cdot) $ be the numerical approximations of $ \psi(\cdot, t_n) $ for $ n \geq 0 $. Then the first-order Lie-Trotter time-splitting Fourier spectral (LTFS) method \cite{bao2023_improved} reads
	\begin{equation}\label{eq:LTFS_scheme}
		\begin{aligned}
			&\psi^{(1)}(x) = e^{- i \tau (V(x) + f(|\psi^n (x)|^2))} \psi^n(x), \\
			&\psi^{n+1}(x) = \sum_{l \in \mathcal{T}_N} e^{- i \tau \mu_l^2} \widehat{(\psi^{(1)})}_l e^{i \mu_l(x-a)}, 
		\end{aligned}
		\quad x \in \Omega, \qquad n \geq 0, 
	\end{equation}
	where $ \psi^0 = P_N \psi_0 $ in \eqref{eq:LTFS_scheme}.

	Under low regularity assumptions on potential, optimal error bounds on the LTFS method are established very recently in \cite{bao2023_improved}, and we recall the results here. Let $0<T<T_\text{max}$ with $T_\text{max}$ being the maximum existing time of the solution $\psi$ to \eqref{NLSE}. Under the assumptions that $V \in L^\infty(\Omega)$ and $\psi \in C([0, T]; H^2_\text{per}(\Omega)) \cap C^1([0, T]; L^2(\Omega))$, when $0<h<h_0$ for some $h_0$ sufficiently small and $\tau \leq h^2/\pi$, we have
	\begin{equation}\label{thm:LTFS}
		\| \psi(\cdot, t_n) - \psin{n} \|_{L^2} \lesssim \tau + h^2, \quad \| \psi(\cdot, t_n) - \psin{n} \|_{H^1} \lesssim \tau^\frac{1}{2} + h, \quad 0 \leq n \leq T/\tau. 
	\end{equation}
	In addition, if $ V \in W^{1, 4}(\Omega) \cap H^1_\text{\rm per}(\Omega) $ and $ \psi \in C([0, T]; H^3_\text{\rm per}(\Omega)) \cap C^1([0, T]; H^1(\Omega)) $, we have
	\begin{equation}\label{thm:LTFS_2}
		\| \psi(\cdot, t_n) - \psin{n} \|_{H^1} \lesssim \tau + h^2, \quad 0 \leq n \leq T/\tau. 
	\end{equation}
	
	However, the LTFS method cannot be efficiently implemented due to a lot of integrals in \eqref{eq:hat} involved in the computation of $\widehat{(\psi^{(1)})}_l$ in \eqref{eq:LTFS_scheme}. In practice, one usually approximate these integrals by some quadrature rules. For any integer $M \geq N$, define $M$ equally distributed quadrature points as
	\begin{equation*}
		x^M_j = a + j \frac{b-a}{M}, \qquad j \in \mathcal{T}_M^0. 
	\end{equation*}
	Then, for $u \in L^\infty(\Omega)$ bounded on $\overline{\Omega}$, the $l$-th Fourier coefficient $\widehat{u}_l$ can be approximated by
	\begin{equation}\label{eq:quadrature}
		\widehat{u}_l \approx \widetilde{u}_{M, l}:= \frac{1}{M} \sum_{j=0}^{M-1} u(x^M_j) e^{-i\mu_l(x^M_j - a)}, \quad l \in \mathcal{T}_N.  
	\end{equation} 
	Note that $\widetilde{u}_{M, l} = \widetilde{v}_l$ for $l \in \mathcal{T}_N$, where
	\begin{equation}
		v = (v_0, \cdots, v_M)^T \in Y_M, \quad v_j = u(x^M_j), \quad j \in \mathcal{T}_M^0. 
	\end{equation}
	
	By applying the approximation $\widehat{(\psi^{(1)})}_l \approx \widetilde{(\psi^{(1)})}_{M, l}$ for $l \in \mathcal{T}_N$ with $M \geq N$ in \eqref{eq:LTFS_scheme}, we obtain the Fourier spectral method with quadrature (FSwQ), which leads to the LTFSwQ method as
	\begin{equation}\label{eq:LTFSq_scheme}
		\begin{aligned}
			&\psi^{(1)}(x) = e^{- i \tau (V(x) + f(|\psi^n_M (x)|^2))} \psi^n_M(x), \\
			&\psi^{n+1}_M(x) = \sum_{l \in \mathcal{T}_N} e^{- i \tau \mu_l^2} \widetilde{(\psi^{(1)})}_{M, l} e^{i \mu_l(x-a)}, 
		\end{aligned}
		\quad x \in \Omega, \qquad n \geq 0, 
	\end{equation}
	with $\psi^0_M \in X_N$ given by
	\begin{equation}\label{eq:LTFSq_ini}
		\psi^0_M(x) = \sum_{l \in \mathcal{T}_N} e^{- i \tau \mu_l^2} \widetilde{(\psi_0)}_{M, l} e^{i \mu_l(x-a)}, \quad x \in \overline{\Omega}.  
	\end{equation}
	Note that the computational time at each time step is at $O(M \log M)$ with the use of FFT. 
	
	If one let $M$ go to infinity in \eqref{eq:LTFSq_scheme}-\eqref{eq:LTFSq_ini}, then the FSwQ method will converge to the FS method used in \eqref{eq:LTFS_scheme}. On the other hand, if one choose $M = N$ in \eqref{eq:LTFSq_scheme}-\eqref{eq:LTFSq_ini}, then the FSwQ method collapses to the standard FP method. 
	
	Compared with the LTFS method \eqref{eq:LTFS_scheme}, the LTFSwQ method \eqref{eq:LTFSq_scheme}-\eqref{eq:LTFSq_ini} is easy to implement. However, it is very hard to establish error estimates of it due to the low regularity of potential $V$. In fact, it is impossible to obtain optimal error bounds in space for the LTFSwQ method with any fixed ratio of $M$ to $N$ (i.e. $M=qN$ with $q$ fixed) as order reduction can be observed numerically (see Section \ref{sec:numerical experiments} and \cite{bao2023_EWI,henning2022}). 
	
	\subsection{An extended Fourier pseudospectral (eFP) method}
	Here, we propose an extended Fourier pseudospectral (eFP) method to discretize the Lie-Trotter splitting in space, which leads to the LTeFP method. The LTeFP method enjoys the benefits of both the LTFS and the LTFSwQ methods: (i) we can establish the same optimal error bounds as the LTFS method analyzed in \cite{bao2023_improved} and (ii) it can be easily implemented as the LTFSwQ method with computational cost the same as the LTFSwQ with $M = 4N$ quadrature points. 
	
	Let $\psihn{n}_j$ be the numerical approximation to $\psi(x_j, t_n)$ for $j \in \mathcal{T}_N^0$ and $n \geq 0$, and denote $\psihn{n}: = (\psihn{n}_0, \psihn{n}_1, \cdots, \psihn{n}_{N})^T \in Y_N$. Then the \textit{LTeFP} method reads
	\begin{equation}\label{LTEFP_scheme}
		\begin{aligned}
			\begin{aligned}
				&\psi^{(1)}(x) = e^{- i \tau V(x)} I_N \left( \psihn{n} e^{-i \tau f(|\psihn{n}|^2)} \right)(x), \quad x \in \Omega \\
				&\psihn{n+1}_j = \sum_{l \in \mathcal{T}_N} e^{- i \tau \mu_l^2} \widehat{(\psi^{(1)})}_l e^{i \mu_l(x_j-a)}, \quad j \in \mathcal{T}_N^0, \quad n \geq 0,  
			\end{aligned} 
		\end{aligned}
	\end{equation}
	where $\psihn{0}_j = \psi_0(x_j)$ for $j \in \mathcal{T}_N^0$. 
	
	In the following, we show how to efficiently compute the Fourier coefficients $\widehat{(\psi^{(1)})}_l$ in \eqref{LTEFP_scheme}. 
	For any $\phi \in Y_N$, we have
	\begin{align*}\label{eq:EFP}
		P_N\left ( e^{- i \tau V} I_N \left( \phi e^{-i \tau f(|\phi|^2)} \right)  \right ) 
		&= P_N\left ( P_{2N} \left (e^{- i \tau V}\right ) I_N \left( \phi e^{-i \tau f(|\phi|^2)} \right)  \right ) \\
		&= P_N I_{4N} \left ( P_{2N} \left (e^{- i \tau V}\right ) I_N \left( \phi e^{-i \tau f(|\phi|^2)} \right)  \right ), 
	\end{align*}
	which implies, for $\psi^{(1)}$ in \eqref{LTEFP_scheme}, 
	\begin{equation}\label{eq:extended_FFT}
		\widehat{(\psi^{(1)})}_l = \widetilde{U}_l, \quad l \in \mathcal{T}_N, 
	\end{equation}
	where $U \in Y_{4N}$ is defined as
	\begin{equation}\label{UjP2n}
		U_j = P_{2N} \left (e^{- i \tau V}\right )(x^{4N}_j) \times I_N \left( \psihn{n} e^{-i \tau f(|\psihn{n}|^2)} \right) (x^{4N}_j), \qquad j \in \mathcal{T}_{4N}^0. 
	\end{equation}
	Hence, as long as $P_{2N}(e^{- i \tau V})$ is precomputed, which can be done either numerically or analytically, the main computational cost of the LTeFP method \eqref{LTEFP_scheme} at each time step comes from applying FFT to a vector with length $4N$.

	\section{Optimal error bounds for the eFP method }\label{sec:LTEFP}
	In this section, we shall establish optimal error bounds for the LTeFP method
	\eqref{LTEFP_scheme} with \eqref{eq:extended_FFT} and \eqref{UjP2n}. Our main results are as follows. 
	
	\subsection{Main results}\label{sec:main_results}
	
	For $ \psihn{n} \ (n \geq 0) $ obtained from the first-order LTeFP method \eqref{LTEFP_scheme}, we have
	\begin{theorem}\label{thm:LTEFP}
		Assume that $ V \in L^\infty(\Omega) $ and $ \psi \in C([0, T]; H^2_\text{\rm per}(\Omega)) \cap C^1([0, T]; L^2(\Omega)) $. There exists $ h_0 > 0 $ sufficiently small such that when $ 0 < h < h_0 $ and $ \tau \leq h^2/\pi $, we have 
		\begin{equation}\label{eq:LT_L2}
			\| \psi(\cdot, t_n) - I_N \psihn{n}  \|_{L^2} \lesssim \tau + h^2, \quad \| \psi(\cdot, t_n) - I_N \psihn{n}  \|_{H^1} \lesssim \sqrt{\tau} + h, \quad 0 \leq n \leq T/\tau. 
		\end{equation}
		In addition, if $ V \in W^{1, 4}(\Omega) \cap H^1_\text{\rm per}(\Omega) $ and $ \psi \in C([0, T]; H^3_\text{\rm per}(\Omega)) \cap C^1([0, T]; H^1(\Omega)) $, we have
		\begin{equation}\label{eq:LT_H1}
			\| \psi(\cdot, t_n) - I_N \psihn{n} \|_{H^1} \lesssim \tau + h^2, \quad 0 \leq n \leq T/\tau.  
		\end{equation}
	\end{theorem}
	

	\begin{remark}[Optimal spatial convergence]\label{rem:spatial_error}
		When assuming higher regularity on the exact solution, one could obtain higher spatial convergence orders in \eqref{eq:LT_L2}-\eqref{eq:LT_H1}, i.e. the spatial error would be at $O(h^m)$ in $L^2$-norm and at $O(h^{m-1})$ in $H^1$-norm if $\psi \in C([0, T]; H^m_\text{per}(\Omega))$ with $ m \geq 2$. In other words, the spatial convergence order of the eFP method is optimal with respect to the regularity of the exact solution. Thus, the eFP method can always achieve optimal spatial convergence orders in practice. 
	\end{remark}
	
	Note that our regularity assumptions on the exact solution are compatible with the assumptions on potential as discussed in Remark 2.3 of \cite{bao2023_improved}. Besides, as shown in \cite{bao2023_improved} as well as the numerical results in Section 5, the time step size restriction $\tau \leq h^2/\pi$ is necessary and optimal. Also, it is natural in terms of the balance between temporal errors and spatial errors.  
	
	In the following, we shall present the proof of Theorem \ref{thm:LTEFP}. With the error estimates of the LTFS method recalled in \eqref{thm:LTFS} and \eqref{thm:LTFS_2}, the error bounds on the LTeFP method can be obtained by directly estimating the error between the solution $\psi^n$ obtained by the LTFS method \eqref{eq:LTFS_scheme} and the solution $\psihn{n}$ obtained by the LTeFP method \eqref{LTEFP_scheme}. We would like to mention that although estimating the error between the LTFS solution and the LTeFP solution seems straightforward, the error between the LTFS solution and the exact solution, which is estimated in \cite{bao2023_improved}, is insightful and non-trivial. 
	
	For ease of presentation, we first define two numerical flows. Let $\Phit{1}:X_N \rightarrow X_N$ be the numerical flow associated with the LTeFP method \eqref{LTEFP_scheme}: 
	\begin{equation}\label{eq:Phi1_def}
		\Phit{1}(\phi) := e^{i \tau \Delta} P_N \left (e^{- i \tau V} I_N \left(\phi e^{- i \tau f(|\phi|^2)}\right) \right ), \qquad \phi \in X_N. 
	\end{equation}
	Then, for $\psihn{n}\ (n\geq 0)$ obtained from \eqref{LTEFP_scheme}, we have
	\begin{equation}\label{LTEFP}
		\begin{aligned}
			&I_N \psihn{n+1} = \Phit{1}(I_N \psihn{n}), \quad n \geq 0, \\
			&I_N \psihn{0} = I_N \psi_0.  
		\end{aligned}
	\end{equation}
	Let $\St{1}:X_N \rightarrow X_N$ be the numerical flow associated with the LTFS method \eqref{eq:LTFS_scheme}:
	\begin{equation}\label{eq:S1}
		\St{1}(\phi) := e^{i \tau \Delta} P_N \Phi_B^\tau(\phi), \qquad \phi \in X_N. 
	\end{equation}
	Then, for $\psi^n \ (n\geq 0)$ obtained from \eqref{eq:LTFS_scheme}, we have
	\begin{equation}\label{eq:LTFS}
		\begin{aligned}
			&\psin{n+1} = \St{1}(\psin{n}), \quad n \geq 1, \\ 
			&\psin{0} = P_N \psi_0. 
		\end{aligned}
	\end{equation}
	
	\subsection{Proof of the optimal $L^2$-norm error bound}
	In this subsection, we shall establish optimal $L^2$-norm error bound \eqref{eq:LT_L2} for the LTeFP method \eqref{LTEFP_scheme}.

	
	From \eqref{thm:LTFS}, by standard projection error estimates of $P_N$ \cite{book_spectral}, we have
	\begin{equation}
		\| P_N \psi(\cdot, t_n) - \psin{n} \|_{L^2} \lesssim \tau + h^2, \quad \| P_N \psi(\cdot, t_n) - \psin{n} \|_{H^1} \lesssim \tau^\frac{1}{2} + h, \quad 0 \leq n \leq T/\tau. 
	\end{equation}
	Using the inverse inequality $\| \phi \|_{H^2} \lesssim h^{-2} \| \phi \|_{L^2}$ for $\phi \in X_N$ \cite{book_spectral} and $\tau \lesssim h^2$, we obtain a uniform $H^2$-norm bound of $\psin{n} \ (0 \leq n \leq /\tau)$ as
	\begin{align}\label{eq:uniform_H2_bound}
		\| \psi^n \|_{H^2} 
		&\lesssim \| \psi^n - P_N \psi(t_n) \|_{H^2} + \| P_N \psi(t_n) \|_{H^2} \notag \\
		&\lesssim h^{-2} \| \psi^n - P_N \psi(t_n) \|_{L^2} + \| \psi(t_n) \|_{H^2} \notag \\
		&\lesssim 1 + \| \psi \|_{L^\infty([0, T]; H^2)}, \qquad 0 \leq n \leq T/\tau. 
	\end{align}
	According to \eqref{eq:uniform_H2_bound}, we define a constant $M_2$ as
	\begin{equation*}
		M_2 := \max_{0 \leq n \leq \frac{T}{\tau}} \{ \| \psin{n} \|_{H^2}, \| \psin{n} \|_{L^\infty} \}. 
	\end{equation*}
	
	\begin{lemma}\label{lem:Hm_bound}
		Let $\phi \in X_N$ and $0 < \tau < 1$. For $m \geq 2$, we have
		\begin{equation*}
			\| \phi (e^{- i \tau f(|\phi|^2)} - 1) \|_{H^m} \leq C(\| \phi \|_{H^m}) \tau.
		\end{equation*}
	\end{lemma}
	
	\begin{proof}
		Noting that $H^m(\Omega)$ with $m \geq 2$ is an algebra, recalling $f(\rho) = \beta \rho$, we have
		\begin{align}
			\| \phi (e^{- i \tau f(|\phi|^2)} - 1) \|_{H^m} 
			&\leq \| \phi \|_{H^m} \| e^{- i \tau f(|\phi|^2)} - 1 \|_{H^m} \leq \| \phi \|_{H^m} \sum_{k=1}^\infty \frac{\tau^k \|f(|\phi|^2)\|_{H^m}^k}{k!} \notag \\
			&\leq \tau \| \phi \|_{H^m} \sum_{k=1}^\infty \frac{\|f(|\phi|^2)\|_{H^m}^k}{k!} \leq \tau \| \phi \|_{H^m} e^{\| f(|\phi|^2) \|_{H^m}} \notag \\
			&\leq \tau \| \phi \|_{H^m} e^{|\beta| \| \phi \|^2_{H^m}} = C(\| \phi \|_{H^m}) \tau, 
		\end{align}
		which completes the proof. 
		\qed 
	\end{proof}
	
	\begin{proposition}[Local truncation error]\label{prop:local_1}
		Let $\phi \in X_N$. Then we have  
		\begin{equation*}
			\| \St{1}(\phi) - \Phit{1}(\phi) \|_{L^2} \leq C(\| \phi \|_{H^2})\tau h^2. 
		\end{equation*}
	\end{proposition}
	
	\begin{proof}
		By the definition of $\St{1}$ and $\Phit{1}$ in \eqref{eq:S1} and \eqref{eq:Phi1_def}, recalling \eqref{eq:nonl_step}, we have
		\begin{align}\label{diff}
			&\St{1}(\phi) - \Phit{1}(\phi) \notag \\
			&= e^{i \tau \Delta} P_N \left( e^{- i \tau V} \left(\phi e^{-i\tau f(|\phi|^2)}\right) \right) - e^{i \tau \Delta} P_N \left( e^{- i \tau V} I_N\left(\phi e^{-i\tau f(|\phi|^2)}\right) \right),	
		\end{align}
		which, by the boundedness of $e^{it\Delta}$ and $P_N$, implies
		\begin{equation}\label{interpolation_error}
			\| \St{1}(\phi) - \Phit{1}(\phi) \|_{L^2} \leq \left \| (I-I_N) \left(\phi e^{-i\tau f(|\phi|^2)} \right) \right \|_{L^2}. 
		\end{equation}
		From \eqref{interpolation_error}, noting that $I_N$ is an identity on $X_N$, by the standard interpolation error estimates of $I_N$ and Lemma \ref{lem:Hm_bound}, we have
		\begin{align*}
			\left \| (I-I_N) \left(\phi e^{-i\tau f(|\phi|^2)} \right) \right \|_{L^2} 
			&= \left \| (I-I_N) \left(\phi (e^{-i\tau f(|\phi|^2)}-1) \right) \right \|_{L^2} \notag\\
			&\lesssim h^2 \|\phi (e^{-i\tau f(|\phi|^2)}-1) \|_{H^2} \leq \tau h^2 C(\| \phi \|_{H^2}), 
		\end{align*}
		which, plugged into \eqref{interpolation_error}, yields the desired result. 
		\qed
	\end{proof}
	
	\begin{proposition}[Stability]\label{prop:stability_1}
		Let $v, w \in X_N$ such that $\| v \|_{L^\infty} \leq M$ and $\| w \|_{L^\infty} \leq M$. Then we have
		\begin{equation*}
			\| \Phit{1}(v) - \Phit{1}(w) \|_{L^2} \leq (1+C(M) \tau) \| v- w \|_{L^2}. 
		\end{equation*}
	\end{proposition}
	
	\begin{proof}
		Recalling \eqref{eq:Phi1_def}, we have
		\begin{equation}\label{eq:diff_stab}
			\Phit{1}(v) - \Phit{1}(w) = e^{i \tau \Delta} P_N \left (e^{- i \tau V} I_N\left(v e^{- i \tau f(|v|^2)} - w e^{- i \tau f(|w|^2)}\right) \right ). 
		\end{equation}
		From \eqref{eq:diff_stab}, by the boundedness of $e^{it\Delta}$ and $P_N$, noting that $I_N$ is an identity on $X_N$, and using
		\begin{equation}
			\| I_N \phi \|_{L^2}^2 = h\sum_{j=0}^{N-1} | (I_N \phi) (x_j) |^2 = h\sum_{j=0}^{N-1} | \phi(x_j) |^2, \quad \phi \in C_\text{per}(\overline{\Omega}), 
		\end{equation}
		we have (see the proof of (4.33) in \cite{bao2023_semi_smooth} for more details)
		\begin{align}\label{eq:diff_stab_L2}
			\| \Phit{1}(v) - \Phit{1}(w) \|_{L^2}
			&\leq \left \| I_N\left(v e^{- i \tau f(|v|^2)} \right) - I_N \left( w e^{- i \tau f(|w|^2)} \right) \right \|_{L^2} \notag\\
			&\leq (1+C(M)\tau) \| v - w \|_{L^2},	
		\end{align}
		which completes the proof. 
		\qed
	\end{proof}
	
	Combining Propositions \ref{prop:local_1} and \ref{prop:stability_1}, using standard Lady-Windermere's fan argument, we can prove \eqref{eq:LT_L2}. Note that we still need to establish a uniform $L^\infty$-bound of the solution $\psihn{n} \  (0 \leq n \leq T/\tau)$ to control the constant in the stability estimate Proposition \ref{prop:stability_1}, which can be done by using mathematical induction with inverse inequalities. We briefly show this process here. 
	\begin{proof}[\eqref{eq:LT_L2}]
		Let $e^n = \psi^n - I_N \psihn{n}$ for $0 \leq n \leq T/\tau$. By \eqref{thm:LTFS}, it suffices to obtain the error bounds for $e^n$. Recalling \eqref{eq:LTFS} and \eqref{LTEFP}, we have
		\begin{align}\label{eq:error_eq}
			e^{n+1} 
			= \psi^{n+1} - I_N \psihn{n+1} = \St{1}(\psi^n) - \Phit{1}(I_N \psihn{n}) \notag \\
			= \St{1}(\psi^n) - \Phit{1}(\psi^n) + \Phit{1}(\psi^n) - \Phit{1}(I_N \psihn{n}). 
		\end{align}
		From \eqref{eq:error_eq}, by triangle inequality and Propositions \ref{prop:local_1} and \ref{prop:stability_1}, we have
		\begin{equation}\label{eq:error_eq_L2}
			\| e^{n+1} \|_{L^2} \leq C(M_2) \tau h^2 + (1 + C(\| I_N \psihn{n} \|_{L^\infty}, M_2)\tau )\| e^n \|_{L^2}. 
		\end{equation}
		We use the induction argument to complete the proof. By standard interpolation and projection error estimates, we have
		\begin{equation}
			\| e^0 \|_{L^2} = \| \psi^0 - I_N \psihn{0} \|_{L^2} = \| P_N \psi_0 - I_N \psi_0 \|_{L^2} \lesssim h^2, \quad \| I_N \psihn{0} \|_{L^\infty} \leq M_2. 
		\end{equation}
		We assume that, for $0 \leq n \leq m \leq T/\tau-1$, 
		\begin{equation}\label{eq:assumption}
			\| e^n \|_{L^2} \lesssim h^2, \quad \| I_N \psihn{n} \|_{L^\infty} \leq 1+M_2. 
		\end{equation}
		We shall show that \eqref{eq:assumption} holds for $n=m+1$. 
		From \eqref{eq:error_eq_L2}, using discrete Gronwall's inequality, noting the $L^\infty$-bound of $I_N \psihn{n} \ (0 \leq  n \leq m)$ in \eqref{eq:assumption}, we have
		\begin{equation}
			\| e^{m+1} \|_{L^2} \lesssim h^2. 
		\end{equation}
		By inverse inequality $\| \phi \|_{L^\infty} \leq h^{-d/2} \| \phi \|_{L^2}$ for $\phi \in X_N$, we have, when $h<h_0$ for some $h_0>0$ small enough, 
		\begin{align}\label{eq:inverse}
			\| I_N \psihn{m+1} \|_{L^\infty} 
			&\leq \| e^{m+1} \|_{L^\infty} + \| P_N \psi^{m+1} - \psin{m+1} \|_{L^\infty} + \| \psin{m+1} \|_{L^\infty} \notag \\
			&\leq C h^{-d/2} \| e^{m+1} \|_{L^2}  + C h^{-d/2} \| P_N \psi^{m+1} - \psin{m+1} \|_{L^2} + M_2 \notag \\
			&\leq C h^{-d/2} h^2  + C h^{-d/2} h^2 + M_2 \leq 1+M_2, 
		\end{align}
		where $d$ is the spatial dimension, i.e. $d=1$ in the current case. Thus, \eqref{eq:assumption} holds for $n = m+1$, and for all $0 \leq n \leq T/\tau$ by mathematical induction, which proves the $L^2$-norm error bound in \eqref{eq:LT_L2} by noting \eqref{thm:LTFS}. The $H^1$-norm error bound in \eqref{eq:LT_L2} follows from
		the inverse inequality $\| \phi \|_{H^1} \lesssim h^{-1} \| \phi \|_{L^2}$ for $\phi \in X_N$ as 
		\begin{equation}
			\| e^n \|_{H^1} \lesssim h^{-1} \| e^n \|_{L^2} \lesssim h, \quad 0 \leq n \leq T/\tau, 
		\end{equation}
		which completes the proof. 
		\qed
	\end{proof}
	
	\subsection{Proof of the optimal $H^1$-norm error bound}
	In this subsection, we shall establish the optimal $H^1$-norm error bound \eqref{eq:LT_H1} under the assumptions that $ V \in W^{1, 4}(\Omega) \cap H^1_\text{per}(\Omega) $ and $ \psi \in C([0, T]; H^3_\text{\rm per}(\Omega)) \cap C^1([0, T]; H^1(\Omega)) $. Similar to the previous section, it suffices to estimate the error between $\psin{n}$ obtained from the LTFS method \eqref{eq:LTFS_scheme} and $\psihn{n}$ obtained from the LTeFP method \eqref{LTEFP_scheme}. 
	
	Similar to \eqref{eq:uniform_H2_bound}, by \eqref{thm:LTFS_2}, we have the following uniform $H^3$-bound of $\psi^n \ (0 \leq n \leq T/\tau)$ obtained from \eqref{LTEFP_scheme}: 
	\begin{equation}\label{eq:uniform_H3_bound}
		\| \psi^n \|_{H^3} \lesssim 1, \quad 0 \leq n \leq T/\tau. 
	\end{equation}
	
	\begin{proposition}[Local truncation error]\label{prop:local_2}
		Let $\phi \in X_N$ and $0<\tau<1$. When $V \in W^{1, 4}(\Omega) \cap H^1_\text{per}(\Omega)$ and $\sigma \geq 1$, we have
		\begin{equation*}
			\| \St{1}(\phi) - \Phit{1}(\phi) \|_{H^1} \leq C(\| V \|_{W^{1, 4}}, \| \phi \|_{H^3}) \tau h^2. 
		\end{equation*}
	\end{proposition}
	
	\begin{proof}
		From \eqref{diff}, using the boundedness of $e^{it\Delta}$ and $P_N$, standard projection error estimates of $I_N$ \cite{book_spectral}, Lemma \ref{lem:Hm_bound}, and the product estimate
		\begin{equation}\label{eq:product_est}
			\| v w \|_{H^1} \lesssim \| v \|_{W^{1, 4}} \| w \|_{H^1}, 
		\end{equation}
		noting that $I_N$ is an identity on $X_N$, we have
		\begin{align*}
			\| \St{1}(\phi) - \Phit{1}(\phi) \|_{H^1} 
			&\leq \left \| e^{-i\tau V}(I-I_N) \left(\phi e^{-i\tau f(|\phi|^2)} \right) \right \|_{H^1} \\
			&\lesssim \| e^{-i\tau V} \|_{W^{1, 4}} \left \| (I-I_N) \left(\phi e^{-i\tau f(|\phi|^2)} \right) \right \|_{H^1} \\
			&\leq C(\| V \|_{W^{1, 4}}) \left \| (I-I_N) \left (\phi (e^{-i\tau f(|\phi|^2)} - 1)\right ) \right \|_{H^1} \\
			&\leq C(\| V \|_{W^{1, 4}}) h^2 \| \phi (e^{-i\tau f(|\phi|^2)} - 1) \|_{H^3} \notag \\
			&\leq \tau h^2 C(\| V \|_{W^{1, 4}}, \| \phi \|_{H^3}),  
		\end{align*}
		which completes the proof. 
		\qed
	\end{proof}
	
	\begin{proposition}[Stability]\label{prop:stability_2}
		Let $0<\tau<1$ and $v, w \in X_N$ such that $\| v \|_{L^\infty} \leq M $, $\| w \|_{L^\infty} \leq M$ and $\| v \|_{H^3} \leq M_1$. When $ V \in W^{1, 4}(\Omega) \cap H^1_\text{per}(\Omega)$, we have 
		\begin{equation*}
			\| \Phit{1}(v) - \Phit{1}(w) \|_{H^1} \leq (1+C(\| V \|_{W^{1, 4}}, M, M_1) \tau) \| v- w \|_{H^1}. 
		\end{equation*}
	\end{proposition}
	
	\begin{proof}
		From \eqref{eq:diff_stab}, by the boundedness of $e^{i\tau\Delta}$ and $P_N$, we have
		\begin{align}\label{eq:diff_stab_H1}
			\| \Phit{1}(v) - \Phit{1}(w) \|_{H^1} 
			&\leq \left \| e^{- i \tau V} I_N\left(v e^{- i \tau f(|v|^2)} - w e^{- i \tau f(|w|^2)} \right) \right \|_{H^1} \notag \\
			&= \left \| e^{- i \tau V} W \right \|_{H^1}, 
		\end{align}
		where 
		\begin{equation}
			W = I_N\left(v e^{- i \tau f(|v|^2)} - w e^{- i \tau f(|w|^2)} \right). 
		\end{equation}
		By triangle inequality, H\"older's inequality and Sobolev embedding $H^1 \hookrightarrow L^4$, we have
		\begin{align}\label{eq:grad_sep}
			\left \| e^{- i \tau V} W \right \|_{H^1}^2 
			&=\left \| W \right \|_{L^2}^2 + \left \| (\nabla W - i \tau W \nabla V) e^{-i \tau V} \right \|_{L^2}^2 \notag \\
			&\leq \| W \|^2_{H^1} + 2 \tau \| \nabla W \|_{L^2}\| W \nabla V \|_{L^2} + \tau^2 \| W \nabla V \|_{L^2}^2 \notag \\
			&\leq \| W \|^2_{H^1} + 2 \tau \| W \|_{H^1}\| W \|_{L^4} \| \nabla V \|_{L^4} + \tau^2 \| W \|_{L^4}^2 \| \nabla V \|_{L^4}^2 \notag \\
			&\leq \| W \|^2_{H^1} + C_1(\| V \|_{W^{1, 4}}) \tau \| W \|_{H^1}^2 + C_2(\| V \|_{W^{1, 4}})\tau^2 \| W \|_{H^1}^2,  
		\end{align}
		which implies from \eqref{eq:diff_stab_H1} that 
		\begin{equation}\label{eq:reduce}
			\| \Phit{1}(v) - \Phit{1}(w) \|_{H^1} \leq \left \| e^{- i \tau V} W \right \|_{H^1} \leq (1 + C(\| V \|_{W^{1, 4}}) \tau ) \| W \|_{H^1}. 
		\end{equation}
		The estimate of $\| W \|_{H^1}$ can be obtained by using the finite difference operator as in Proposition 4.8 of \cite{bao2023_semi_smooth}. In fact, following the proof for (4.34) in \cite{bao2023_semi_smooth} with $V(x) \equiv 0$, we have
		\begin{equation}\label{eq:W_est}
			\| W \|_{H^1} \leq (1 + C(M, M_1)\tau) \| v- w \|_{H^1}, 
		\end{equation} 
		which plugged into \eqref{eq:reduce} completes the proof. 
		\qed 
	\end{proof}
	
	\begin{remark}
		Note that, in Proposition \ref{prop:stability_2}, we have stronger assumption $\| v \|_{H^3} \leq M_1$ instead of $\| v \|_{H^2} \leq M_1$ assumed in Proposition 4.8 of \cite{bao2023_semi_smooth}, which will simplify the proof of Proposition 4.8 of \cite{bao2023_semi_smooth} in 2D and 3D. 
	\end{remark}
	
	Following the proof of \eqref{eq:LT_L2}, the proof of \eqref{eq:LT_H1} can be completed by directly using Lady-Windermere's fan argument with Propositions \ref{prop:local_2} and \ref{prop:stability_2}. Since we have established a uniform $L^\infty$-bound of $I_N \psihn{n} \ (0 \leq n \leq T/\tau)$ in the proof of \eqref{eq:LT}, the constant $C$ in Proposition \ref{prop:stability_2} is already under control, and thus the induction argument is not needed. We do not detail the process here for brevity. 
	
	
	\begin{remark}[Generalization to the Strang splitting]\label{rem:Strang}
		By using the eFP method to discretize the second-order Strang splitting in space, one can similarly obtain the Strang time-splitting extended Fourier pseudospectral (\textit{STeFP}) method
		\begin{equation}\label{STEFP_scheme}
			\begin{aligned}
				&\psi^{(1)}_j = \sum_{l \in \mathcal{T}_N} e^{- i \frac{\tau}{2} \mu_l^2} \widetilde{(\psihn{n})}_l e^{i \mu_l(x_j-a)}, \quad j \in \mathcal{T}_N^0, \\
				&\psi^{(2)}(x) = e^{- i \tau V(x)} I_N \left( \psi^{(1)} e^{-i \tau f(|\psi^{(1)}|^2)} \right)(x), \quad x \in \Omega, \\
				&\psihn{n+1}_j = \sum_{l \in \mathcal{T}_N} e^{- i \frac{\tau}{2} \mu_l^2} \widehat{(\psi^{(2)})}_l e^{i \mu_l(x_j-a)}, \quad j \in \mathcal{T}_N^0,  
			\end{aligned}
			\quad n \geq 0, 
		\end{equation}
		where $\psihn{0}_j = \psi_0(x_j)$ for $j \in \mathcal{T}_N^0$. The computation of $\widehat{(\psi^{(2)})}_l$ can be done in a manner similar to that presented in \eqref{eq:extended_FFT} for $\widehat{(\psi^{(1)})}_l$. 
		The optimal error bounds on the STeFP method also follow immediately: For $ \psihn{n} \ (n \geq 0) $ obtained from the second-order STeFP method \eqref{STEFP_scheme}, we have
		\begin{theorem}\label{thm:STEFP}
			Assume that $ V \in H_\text{\rm per}^2(\Omega) $ and $ \psi \in C([0, T]; H^4_\text{\rm per}(\Omega)) \cap C^1([0, T]; H^2(\Omega)) \cap C^2([0, T]; L^2(\Omega)) $. There exists $ h_0 > 0 $ sufficiently small such that when $ 0 < h < h_0 $ and $ \tau \leq h^2/\pi $, we have 
			\begin{equation}\label{eq:ST_L2}
				\| \psi(\cdot, t_n) - I_N \psihn{n} \|_{L^2} \lesssim \tau^2 + h^4, \quad \| \psi(\cdot, t_n) - I_N \psihn{n} \|_{H^1} \lesssim \tau^\frac{3}{2} + h^3, \quad 0 \leq n \leq T/\tau.  
			\end{equation}
			In addition, if $ V \in H^3_\text{\rm per}(\Omega) $ and $ \psi \in C([0, T]; H^5_\text{\rm per}(\Omega)) \cap C^1([0, T]; H^3(\Omega)) \cap C^2([0, T]; H^1(\Omega)) $, we have 
			\begin{equation}\label{eq:ST_H1}
				\| \psi(\cdot, t_n) - I_N \psihn{n} \|_{H^1} \lesssim \tau^2 + h^4, \quad 0 \leq n \leq T/\tau. 
			\end{equation}
		\end{theorem}
	\end{remark}
	
	\begin{remark}[Generalization to other nonlinearity]
		The results in the current paper can be easily generalized to the GPE \eqref{NLSE} with more general nonlinearity of the form $f(|\psi|^2)\psi$ with $f(\rho) = \beta \rho^{\sigma}$ or $f(\rho) = \beta \rho^{\sigma} \ln \rho$ under suitable assumptions on $\sigma$ as considered in \cite{bao2023_improved,bao2023_EWI,bao2023_semi_smooth}. 
	\end{remark}
	
	\section{Numerical results}\label{sec:numerical experiments}
	In this section, we present some numerical results of applying time-splitting eFP methods \eqref{LTEFP_scheme} and \eqref{STEFP_scheme} to solve the GPE \eqref{NLSE} with low regularity potential. In the following, we fix $ d=1 $, $ \Omega = (-16, 16) $, $ T=1 $ and choose a Gaussian type initial datum
	\begin{equation}\label{eq:ini}
		\psi_0(x) = e^{-x^2/2}, \quad x \in \Omega. 
	\end{equation}
	To quantify the error, we introduce the following error functions:
	\begin{align*}
		e_{L^2}(t_n) := \| \psi(\cdot, t_n) - \psi^n  \|_{L^2}, \quad  e_{H^1}(t_n) := \| \psi(\cdot, t_n) - \psi^n \|_{H^1}, \quad 0 \leq n \leq T/\tau. 
	\end{align*}
	We consider four potential functions $V_j \ ( j  = 1, 2, 3, 4) $ of different regularities given by
	\begin{equation}\label{eq:poten}
		\begin{aligned}
			&V_1(x) = \left\{
			\begin{aligned}
				&0, &x \in (-4, 4) \\
				&10, &\text{otherwise}
			\end{aligned}
			\right.
			, \quad &&V_2(x) = |x|^{0.76}, \\
			&V_3(x) = |x|^{1.51} \left (1- \frac{x^2}{16^2} \right )^2, \quad &&V_4(x) = |x|^{2.51} \left (1- \frac{x^2}{16^2} \right )^3, 
		\end{aligned}
		\qquad x \in \Omega. 
	\end{equation}
	Note that the potential functions $V_j \ ( j  = 1, 2, 3, 4)$ defined in \eqref{eq:poten} satisfy $ V_1 \in L^\infty(\Omega) $, $V_2 \in W^{1, 4}(\Omega) \cap H^1_\text{per}(\Omega)$, $ V_3 \in H^2_\text{per}(\Omega) $ and $ V_4 \in H^3_\text{per}(\Omega) $. 
	
	The `exact' solutions are computed by the STeFP method \eqref{STEFP_scheme} with $ \tau = \tau_\text{e} := 10^{-6} $ and $ h = h_\text{e} := 2^{-9} $.  
	
	\subsection{Spatial errors}
	In this subsection, we show the spatial errors of the eFP method for the GPE with different potentials $V = V_j \ (j = 1, 2, 3, 4)$. We also carry out comparisons with the Fourier spectral method with quadrature (FSwQ) for spatial discretization. 
	
	Since we only care about spatial errors in this subsection, the choice of temporal discretization does not matter, and we only use the Strang splitting for temporal discretization. In computation, we fix $\tau = \tau_{\text e}$ such that the temporal errors are negligible compared to the spatial errors, and choose mesh size $h=1/N$ with $N$ ranging from $2^7$ to $2^{11}$. 
	
	We start with the $L^\infty$-potential $V = V_1$ in \eqref{NLSE}. In Figure \ref{fig:Linfty_poten}, we plot the errors in $L^2$- and $H^1$-norm for the eFP and the FSwQ methods. We use FSwQ-$M$ to indicate that we are applying the FSwQ with $M \geq N$ quadrature points $\{x^M_j\}_{j=0}^{M-1}$. Hence, the computational cost of FSwQ-$M$ is at $O(M \log M)$. In comparison, the computational cost of the eFP method is at $O(N \log N)$.  
	
	We observe that the eFP method converges with $2.5$ order in $L^2$-norm and $1.5$ order in $H^1$-norm. Such convergence orders are not surprising by recalling Remark \ref{rem:spatial_error} since the exact solution in this case has regularity roughly $H^{2.5}$. While, for the FSwQ method, the numerical results suggest that the error is at $O(1/N^{2.5}) + O(1/M)$ in $L^2$-norm and at $O(1/N^{1.5}) + O(1/M)$ in $H^1$-norm. Hence, to obtain the same convergence rate in $L^2$-norm as the eFP method, one shall use FSwQ with $M \sim N^{2.5}$ quadrature points, which is extremely time-consuming. 
	\begin{figure}[htbp]
		\centering
		{\includegraphics[height=\figheight\textheight]{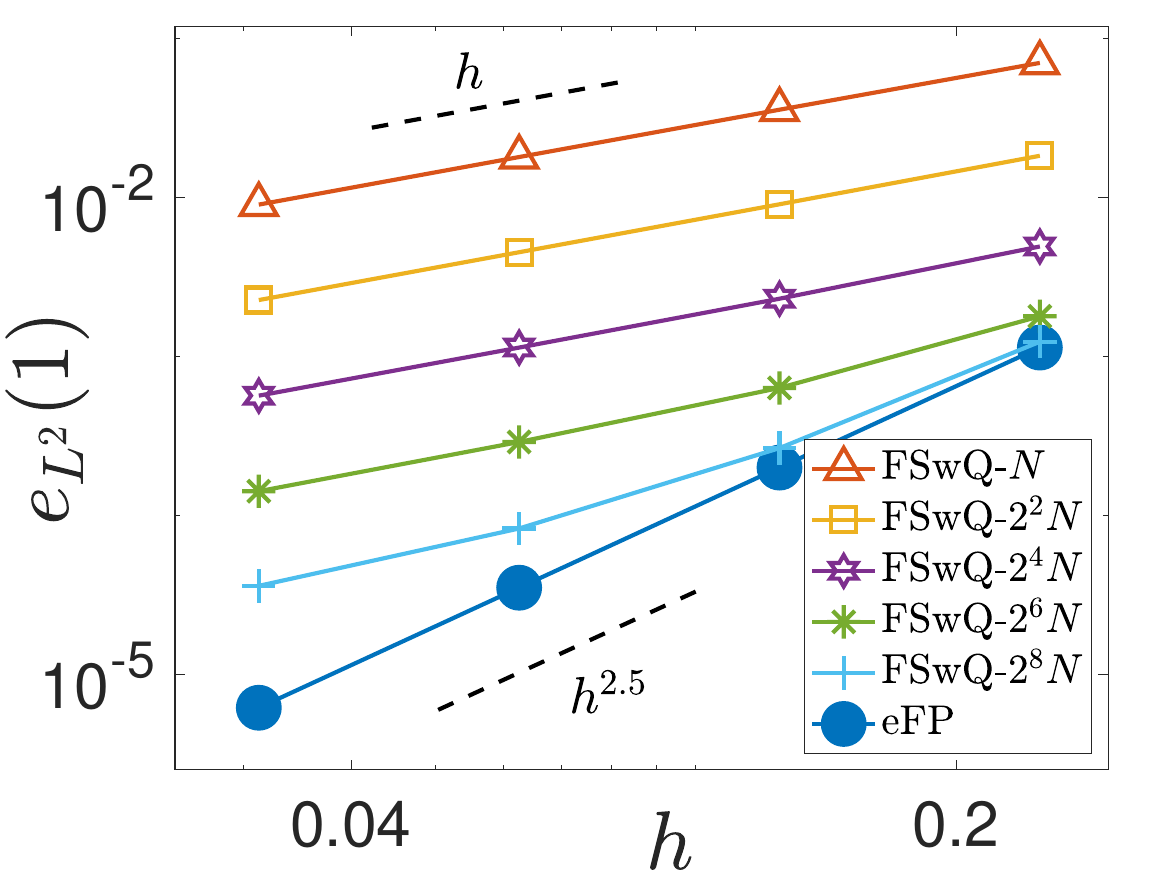}}\hspace{1em}
		{\includegraphics[height=\figheight\textheight]{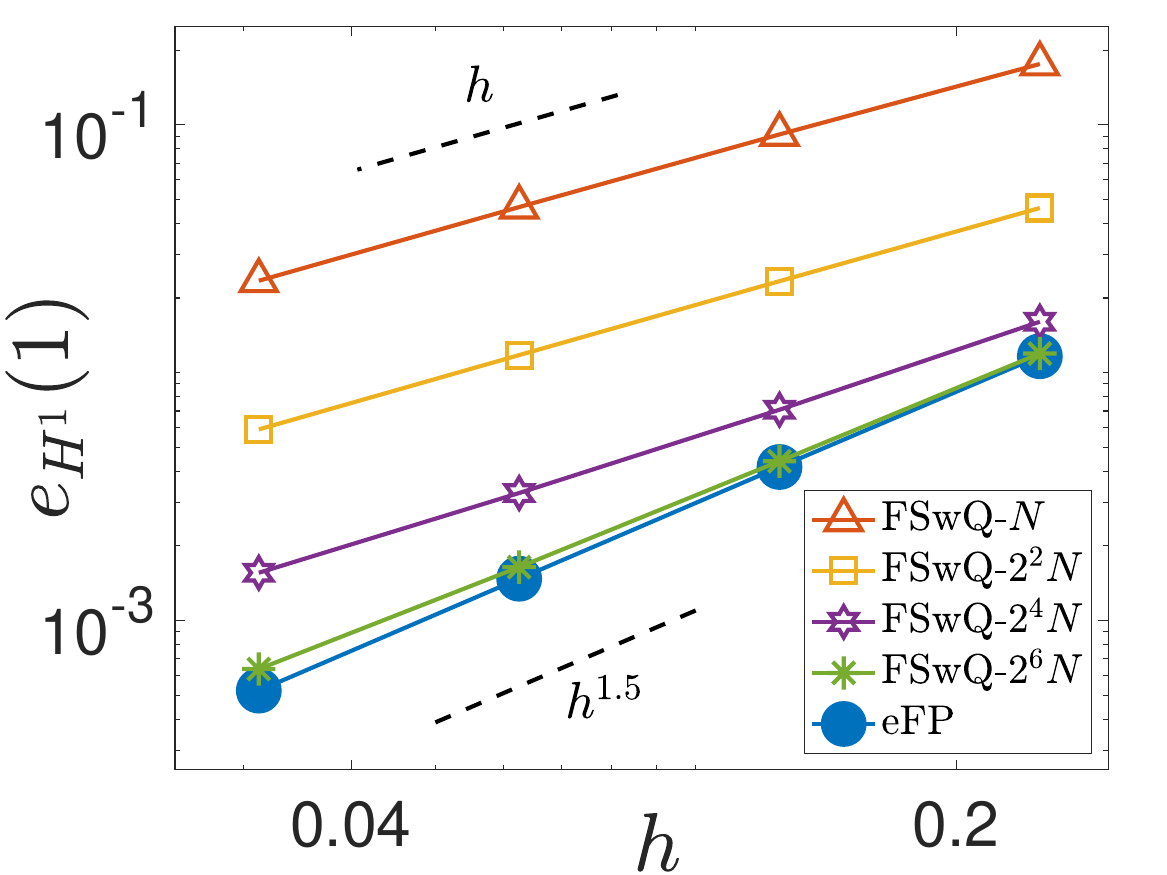}}
		\caption{Spatial errors in $L^2$- and $H^1$-norm of the eFP method for the GPE \eqref{NLSE} with $L^\infty$-potential $V = V_1$.}
		\label{fig:Linfty_poten}
	\end{figure}
	
	In Figures \ref{fig:W14_poten}-\ref{fig:H3_poten}, we exhibit corresponding results for $V = V_2$, $V=V_3$ and $V=V_4$, respectively. Similarly, in all the cases, the eFP method demonstrates optimal spatial convergence orders in both $L^2$- and $H^1$-norm consistent with the regularity of the exact solution. The orders of the eFP method are also higher than the FSwQ method for any fixed ratio of $M$ to $N$ in either $L^2$-norm or $H^1$-norm. However, the advantage of the eFP method diminishes when dealing with high regularity potential compared to its superiority in the presence of low regularity potential. 
	\begin{figure}[htbp]
		\centering
		{\includegraphics[height=\figheight\textheight]{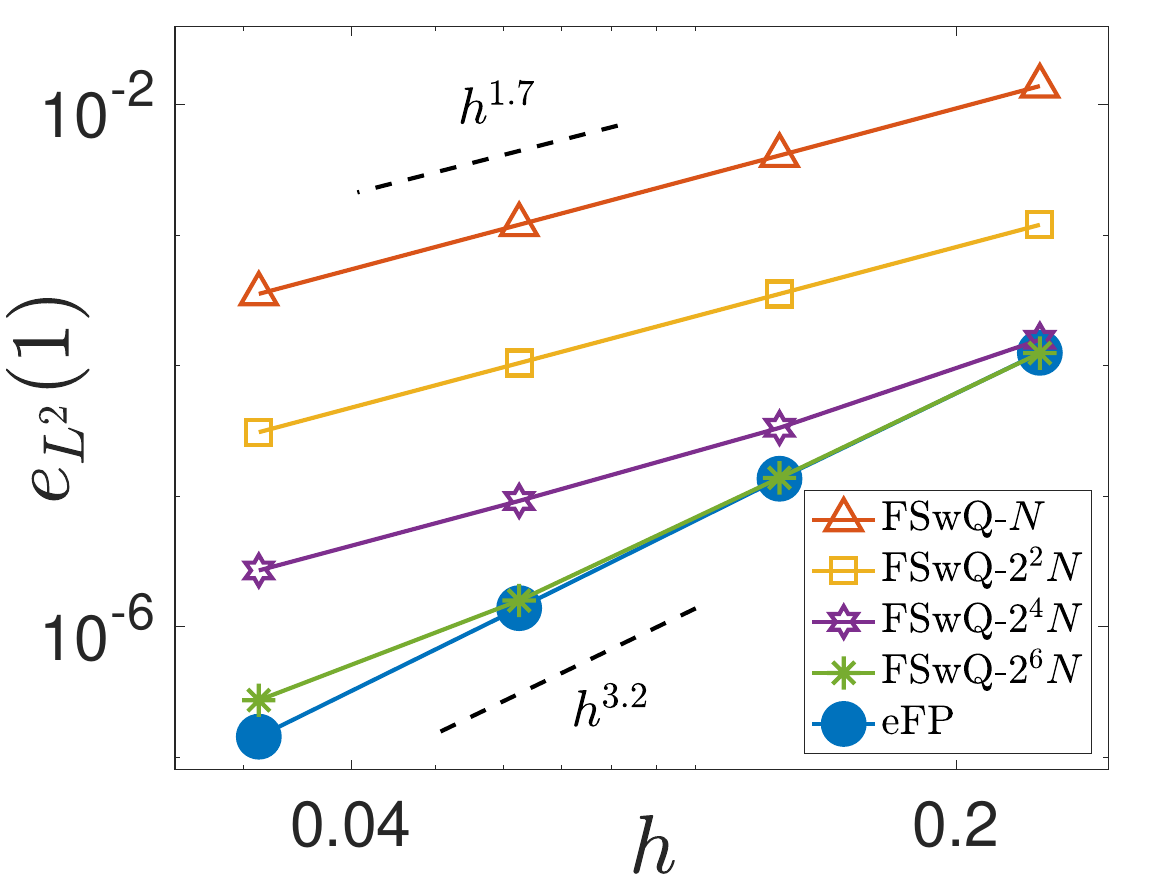}}\hspace{1em}
		{\includegraphics[height=\figheight\textheight]{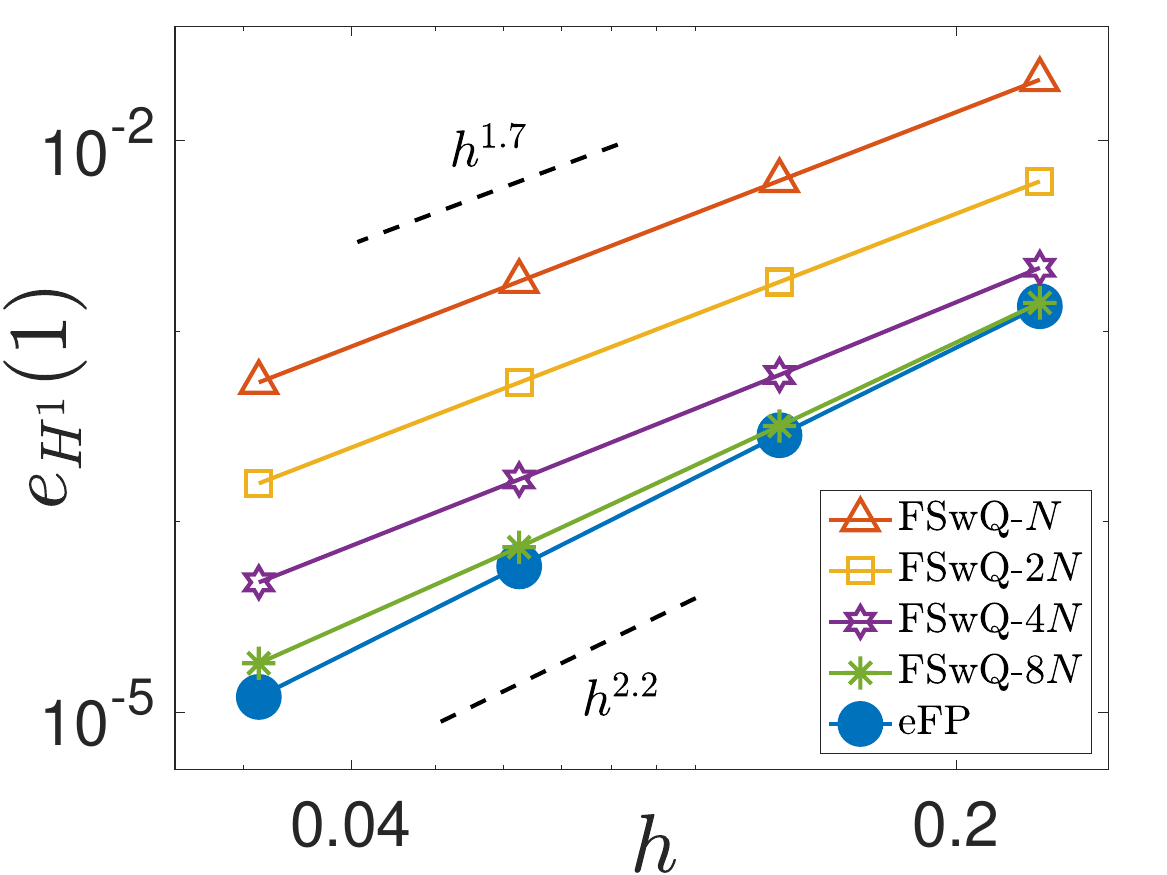}}
		\caption{Spatial errors in $L^2$- and $H^1$-norm of the eFP method for the GPE \eqref{NLSE} with $W^{1, 4}$-potential $V = V_2$.}
		\label{fig:W14_poten}
	\end{figure}
	
	\begin{figure}[htbp]
		\centering
		{\includegraphics[height=\figheight\textheight]{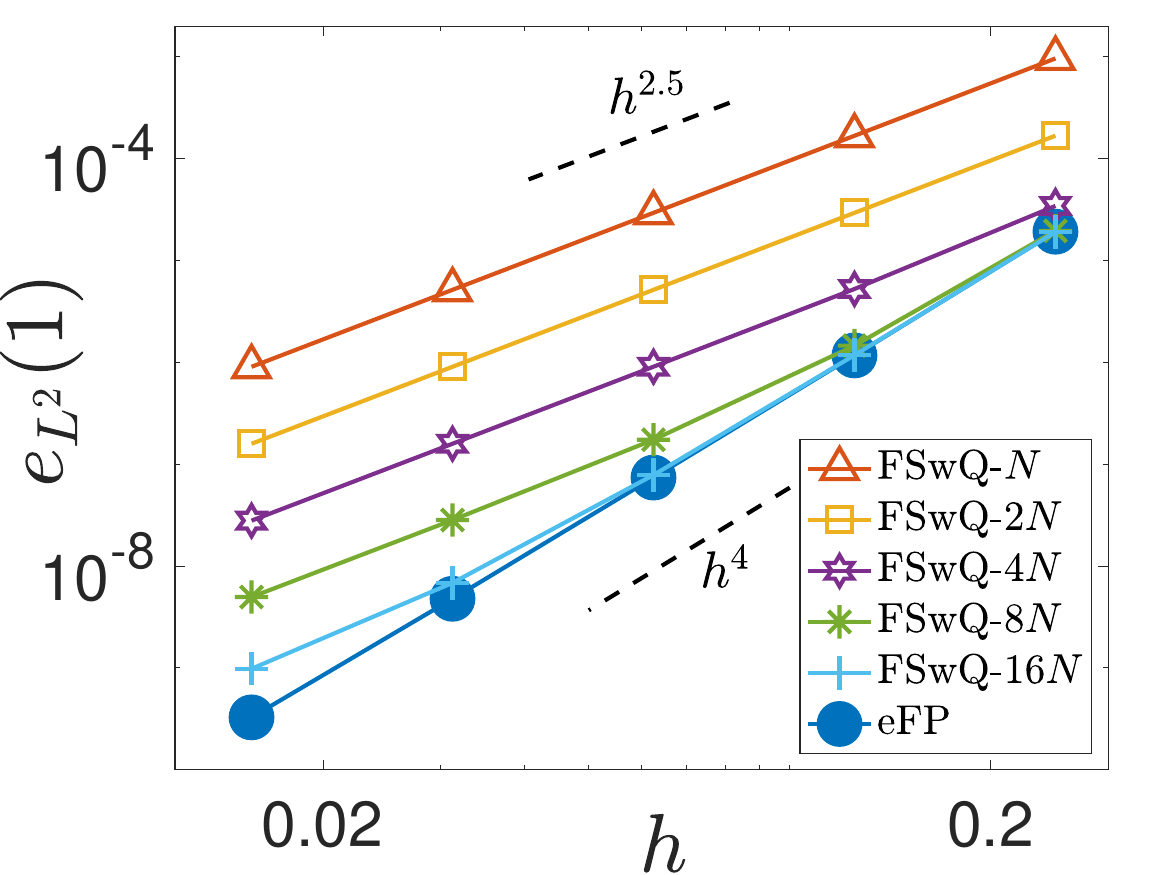}}\hspace{1em}
		{\includegraphics[height=\figheight\textheight]{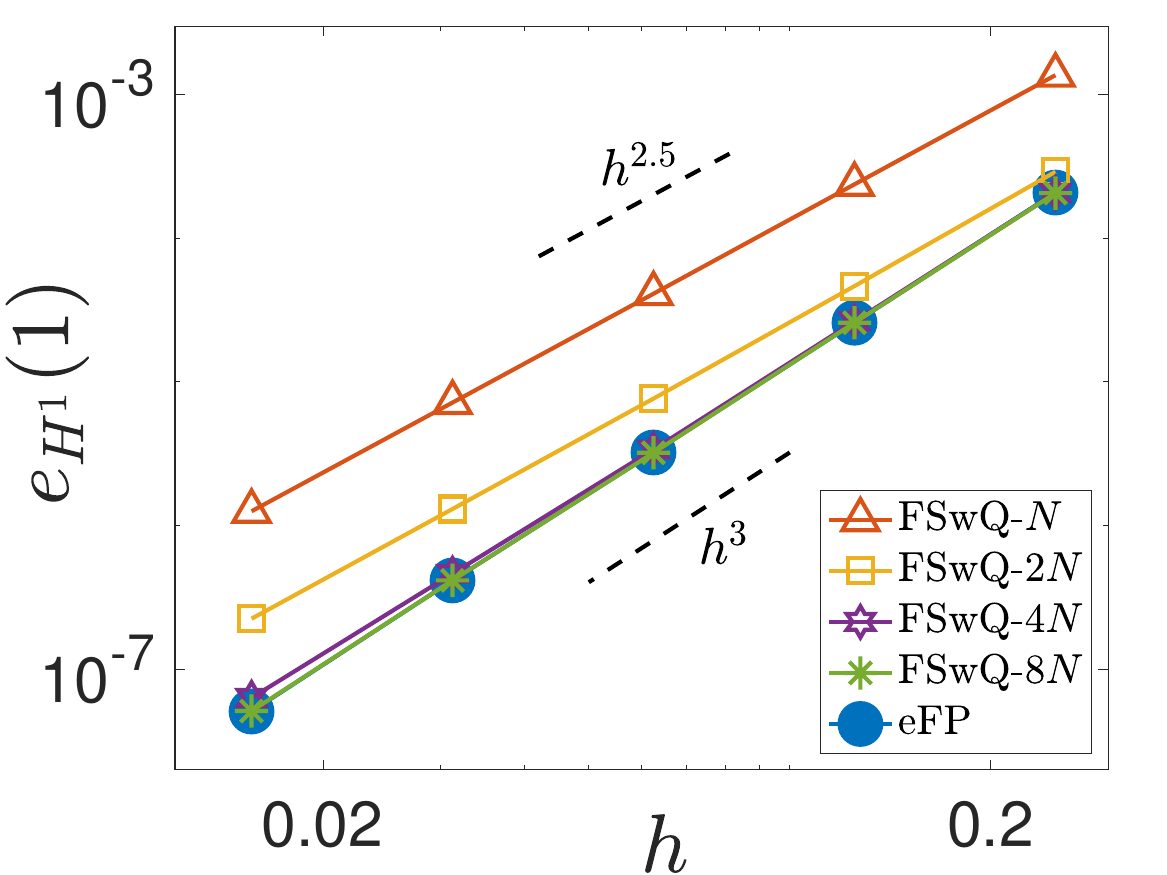}}
		\caption{Spatial errors in $L^2$- and $H^1$-norm of the eFP method for the GPE \eqref{NLSE} with $H^2$-potential $V = V_3$.}
		\label{fig:H2_poten}
	\end{figure}
	
	\begin{figure}[htbp]
		\centering
		{\includegraphics[height=\figheight\textheight]{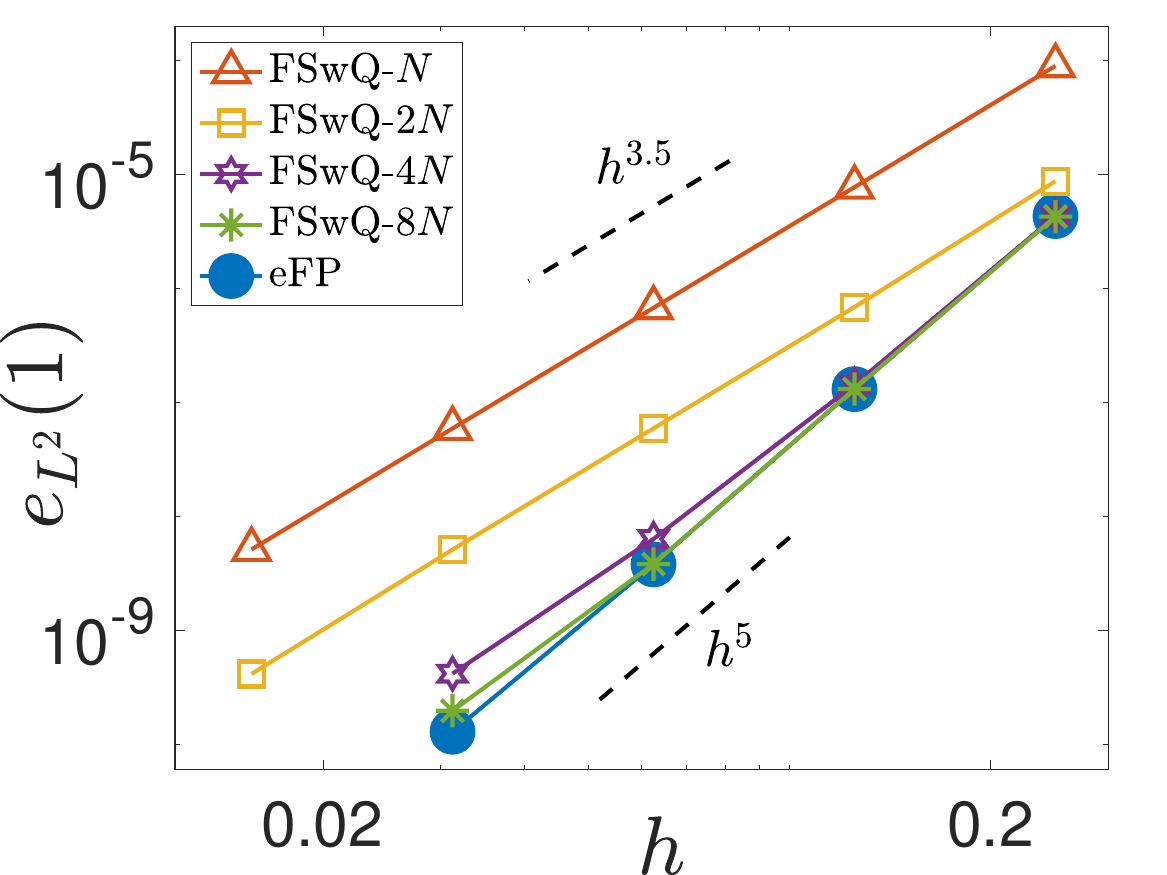}}\hspace{1em}
		{\includegraphics[height=\figheight\textheight]{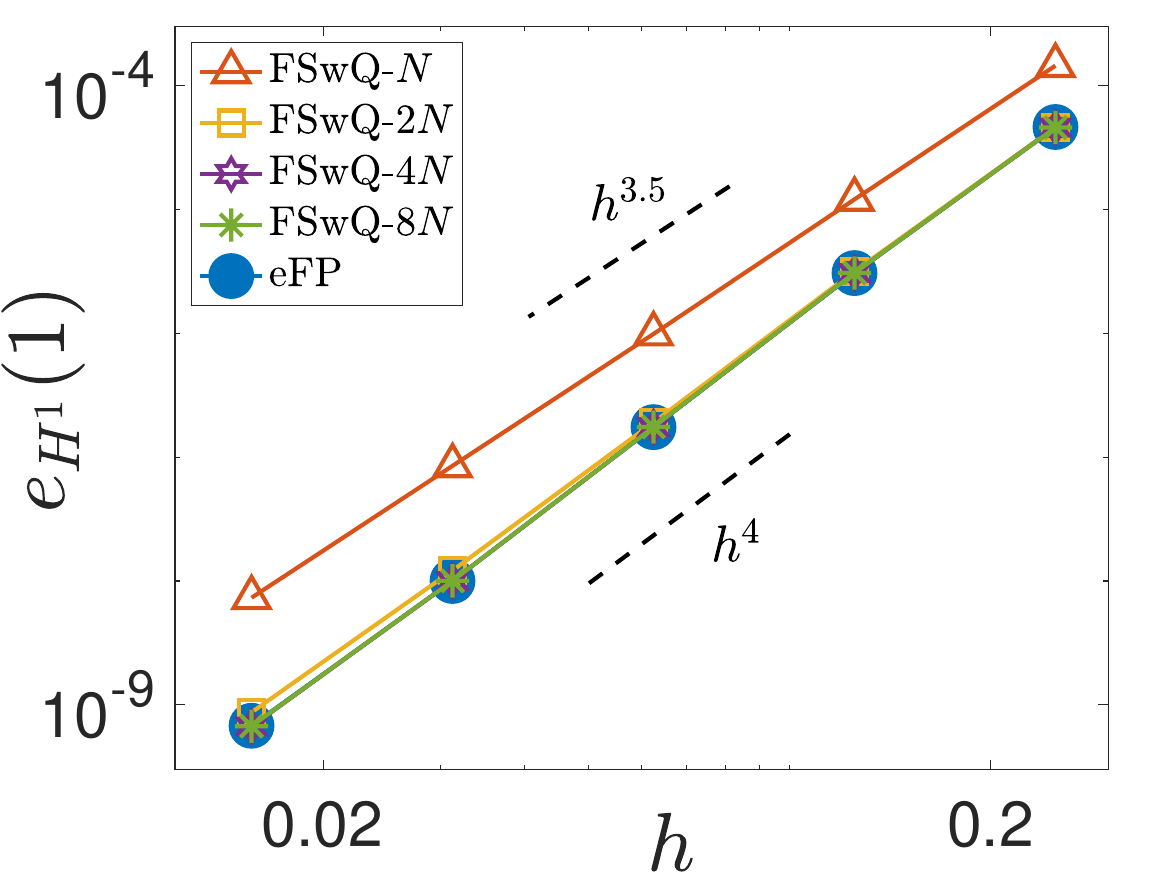}}
		\caption{Spatial errors in $L^2$- and $H^1$-norm of the eFP method for the GPE \eqref{NLSE} with $H^3$-potential $V = V_4$.}
		\label{fig:H3_poten}
	\end{figure}
	
	In summary, the eFP method can achieve optimal spatial convergence with respect to the regularity of the exact solution, which confirms our error estimates in Theorems \ref{thm:LTEFP} and \ref{thm:STEFP} as well as Remark \ref{rem:spatial_error}. It surpasses the FSwQ method, with the Fourier pseudospectral method as a particular case, for both low and high regularity potential. While its superiority is most evident in the cases of low regularity potential.   
	
	\subsection{Temporal errors}
	In this subsection, we shall test the temporal convergence orders of the LTeFP method \eqref{LTEFP_scheme} and the STeFP method \eqref{STEFP_scheme} for the GPE \eqref{NLSE} with different potential $V = V_j \ (j = 1, 2, 3, 4)$. To demonstrate that the time step size restriction $\tau \leq h^2/\pi$ is necessary and optimal, we show numerical results obtained with $\tau = C h^\gamma$ for different $C$ and $\gamma$. 
	
	We start with the LTeFP method and choose $V = V_1$ and $V=V_2$ for optimal first-order $L^2$- and $H^1$-norm error bounds, respectively. The numerical results are shown in Figure \ref{fig:LT_poten}, where the errors are computed with (i) $\tau = 0.2h$, (ii) $\tau = 0.4h^{1.5}$, (iii) $\tau = 0.8h^2$ and (iv) $\tau=0.2h^2$. Note that, from (iv) to (i), we are refining the mesh; however, as we shall present in the following, refining the mesh will significantly increase the temporal error and also the overall error. 
	
	From (a1) and (a2) in Figure \ref{fig:LT_poten}, we see that the expected temporal convergence orders (i.e. first order in $L^2$-norm and half order in $H^1$-norm) proved in \eqref{eq:LT_L2} can only be observed when $\tau \leq h^2/\pi$ (corresponding to $\tau = 0.2h^2$). There is order reduction for other choices of $\tau$ and $h$. In particular, for the common choice of $\tau \sim h$, there is no convergence in $H^1$-norm. Similarly, from (b) in Figure \ref{fig:LT_poten}, the optimal first-order convergence in $H^1$-norm proved in \eqref{eq:LT_H1} can only be observed when $\tau \leq h^2/\pi$. These observations confirm our error bounds in Theorem \ref{thm:LTEFP} for the GPE with low regularity potential and indicate that the time step size restriction $\tau \leq h^2/\pi$ is necessary and optimal. 
	\begin{figure}[htbp]
		\centering
		{\includegraphics[height=\figheight\textheight]{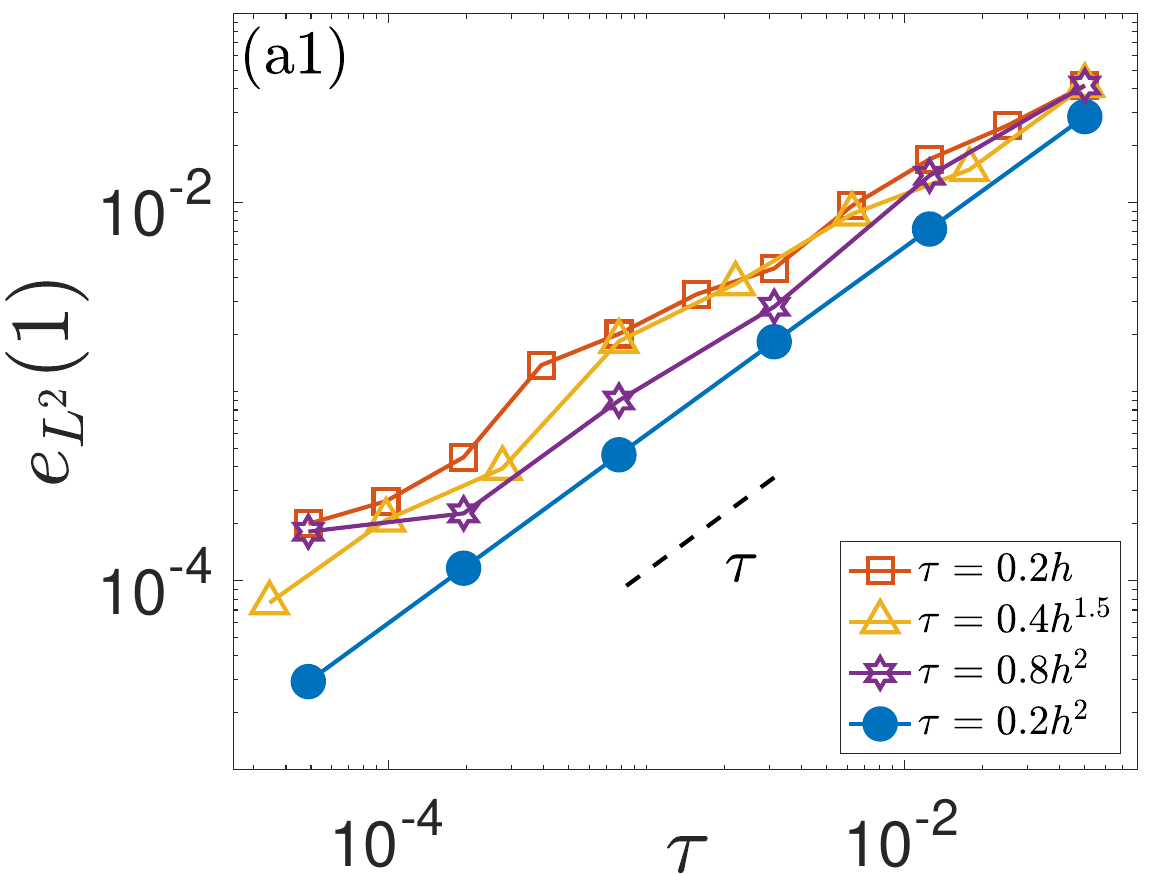}}\hspace{1em}
		{\includegraphics[height=\figheight\textheight]{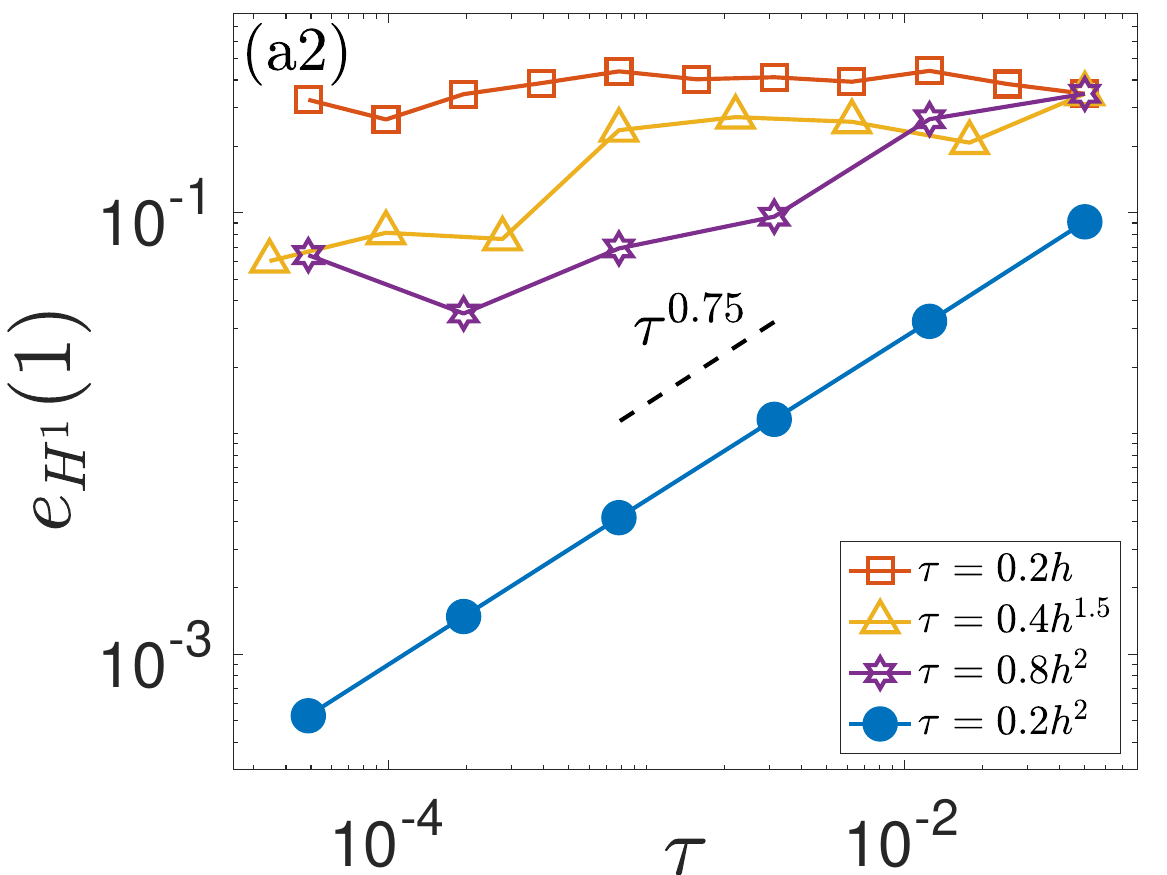}}\\
		{\includegraphics[height=\figheight\textheight]{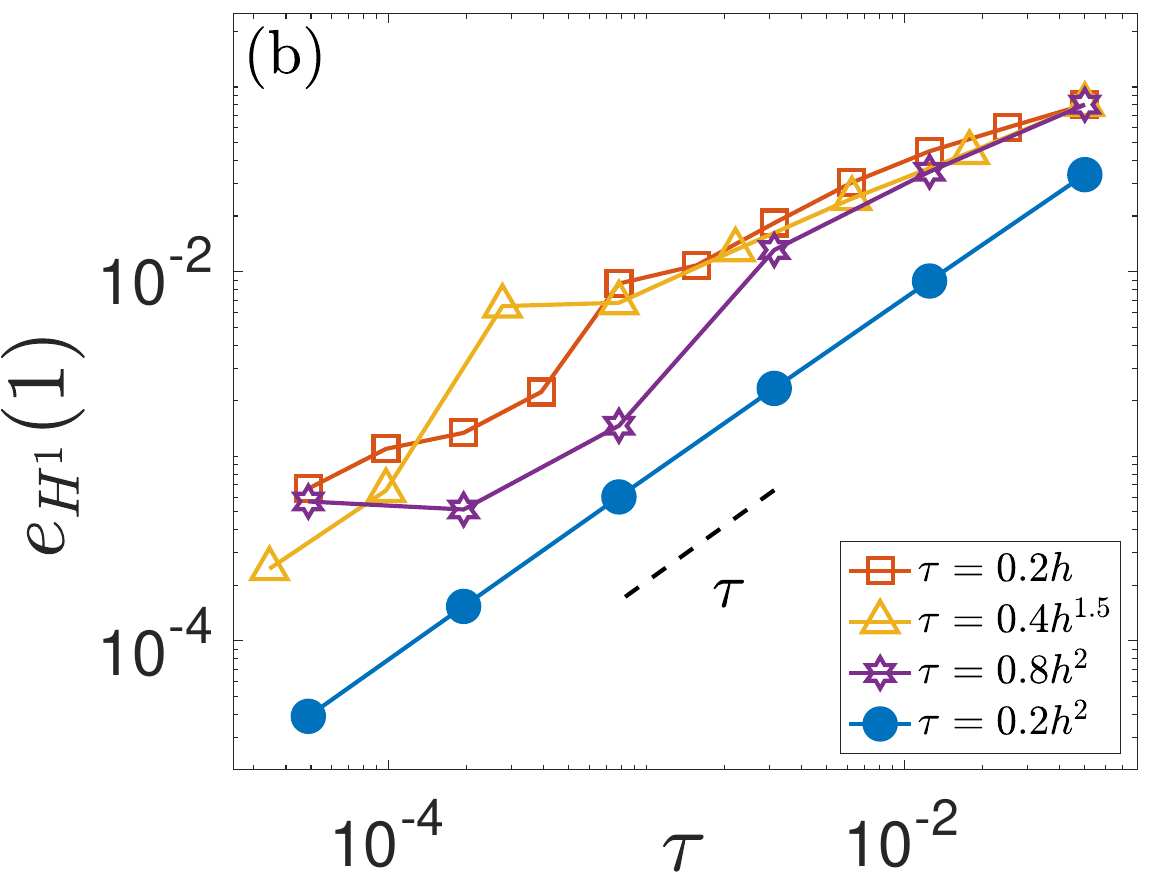}}
		\caption{Temporal errors of the LTeFP method for the GPE \eqref{NLSE} with (a) $L^\infty$-potential $V=V_1$ and (b) $W^{1, 4}$-potential $V = V_2$. }
		\label{fig:LT_poten}
	\end{figure}
	
	Then we showcase the results using the STeFP method for $V = V_3$ and $V = V_4$, aiming for optimal second-order error bounds in the $L^2$- and $H^1$-norms, respectively. The numerical findings are depicted in Figure \ref{fig:ST_poten} with consistent choices for $\tau$ and $h$, namely, $\tau = 0.2h$, $\tau = 0.4h^{1.5}$, $\tau = 0.8h^2$, and $\tau=0.2h^2$. Observations from these results align closely with those from the LTeFP method, validating our optimal error bounds in Theorem \ref{thm:STEFP} and suggesting that the time step size restriction $\tau \leq h^2/\pi$ is necessary and optimal. 
	
	\begin{figure}[htbp]
		\centering
		{\includegraphics[height=\figheight\textheight]{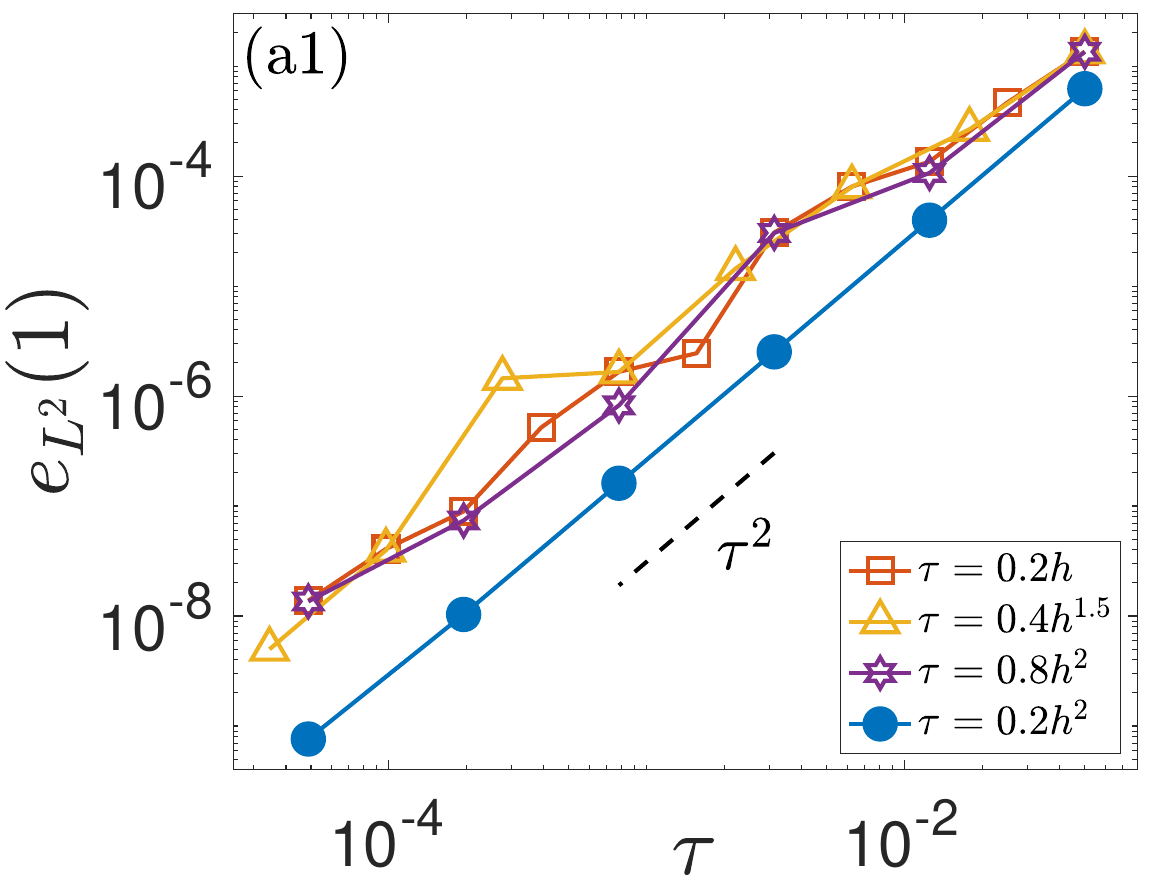}}\hspace{1em}
		{\includegraphics[height=\figheight\textheight]{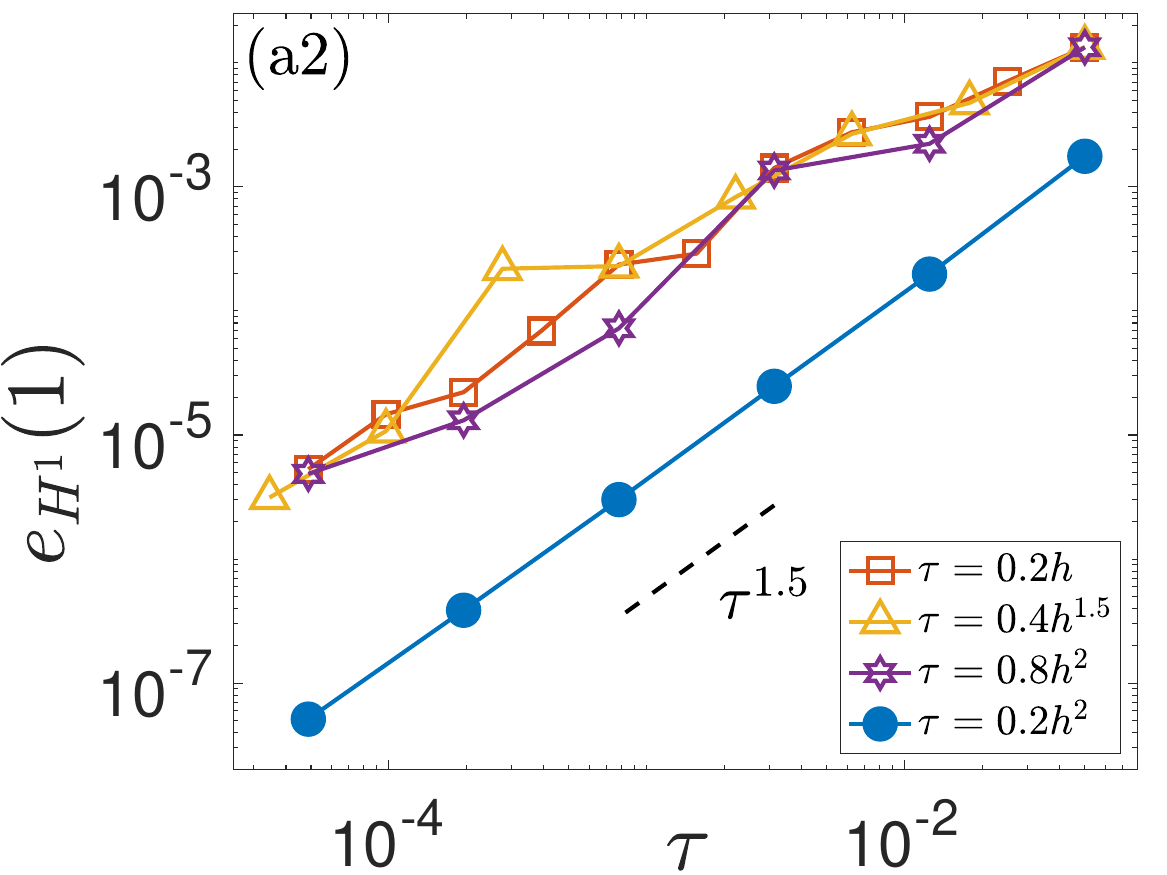}}\\
		{\includegraphics[height=\figheight\textheight]{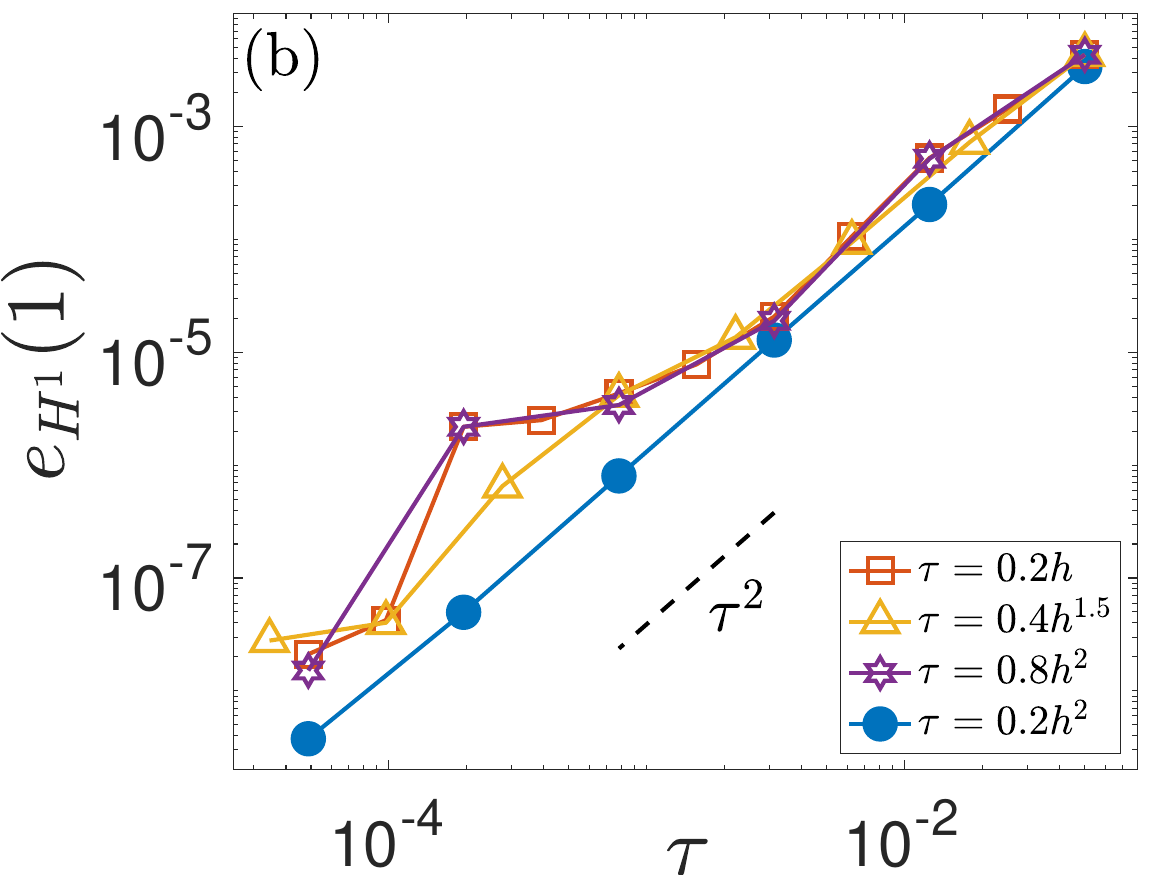}}
		\caption{Temporal errors of the STeFP method for the GPE \eqref{NLSE} with (a) $H^2$-potential $V=V_3$ and (b) $H^3$-potential $V = V_4$. }
		\label{fig:ST_poten}
	\end{figure}
	
	\section{Conclusion}
	We proposed and analyzed an extended Fourier pseudospectral (eFP) method for the spatial discretization of the Gross-Pitaevskii equation with low regularity potential. We also rigorously established error estimates for the fully discrete scheme obtained by combining the eFP method with time-splitting methods. The eFP method is accurate and efficient: it maintains optimal approximation rates with computational cost almost the same as the standard Fourier pseudospectral method. Furthermore, instead of time-splitting methods, it can be coupled with various kinds of temporal discretizations such as finite difference methods and exponential-type integrators, showing its high flexibility. 
	
	\section*{Acknowledgments}
	The work is partially supported by the Ministry of Education of Singapore under its AcRF Tier 2 funding MOE-T2EP20122-0002 (A-8000962-00-00). 
	

\end{document}